\documentclass[11pt,a4paper]{article}
\usepackage{jheppub,amsthm,amsmath,amssymb,slashed,url,bm,mathrsfs}
\usepackage{graphicx}
\usepackage[all]{xy}
\usepackage{epstopdf}
\usepackage{setspace}
\usepackage{verbatim}

\newtheorem{theorem}{Theorem}[section]
\newtheorem{lemma}[theorem]{Lemma}
\newtheorem{definition}[theorem]{Definition}
\newtheorem{proposition}[theorem]{Proposition}

\newtheorem{corollary}[theorem]{Corollary}

\def\smallint{\begingroup\textstyle \int\endgroup}

\def\^{{\wedge}}
\def\*{{\star}}
\def\bar{\overline}

\def\ad#1{{\rm ad}#1}

\def\End{\mathop{\rm End}}
\def\Hol{\mathop{\rm Hol}}
\def\Hom{\mathop{\rm Hom}}

\def\Map{\mathop{\rm Map}}

\def\mod{\mathop{\rm mod}}

\def\Tr{{\mathop{\rm Tr}}}

\def\BC{{\mathbb C}}

\def\BH{{\mathbb H}}

\def\BR{{\mathbb R}}

\def\BT{{\mathbb T}}

\def\BZ{{\mathbb Z}}

\def\CC{{\mathcal C}}

\def\CE{{\mathcal E}}
\def\CF{{\mathcal F}}

\def\CL{{\mathcal L}}
\def\CM{{\mathcal M}}

\def\CO{{\mathcal O}}

\def\CS{{\mathcal S}}

\def\Fg{{\mathfrak g}}
\def\Fh{{\mathfrak h}}

\def\Fk{{\mathfrak k}}

\def\Ft{{\mathfrak t}}

\def\Ri{{\eurm i}}

\def\affgs{{\tilde {\mathfrak{g}}}} 
\def\affg{{\hat {\mathfrak{g}}}} 
\def\flgs{{\check{\mathfrak{g}}}} 
\def\flg{{{L\mathfrak{g}}}} 
\def\LGS{{\check \AG}} 
\def\LG{{LG}} 

\def\lllangle{\langle\!\!\langle} 
\def\rrrangle{\rangle\!\!\rangle}
\def\llangle{\langle\!\langle} 
\def\rrangle{\rangle\!\rangle}

\def\Ad{\mathrm{Ad\,}}
\def\ad{\mathrm{ad\,}}

\def\forms{\Omega} 
\def\CSg{\mathbb{CS}} 
\def\cs{\mathrm{cs}} 
\def\CS{\mathrm{CS}} 
\def\msv{\mathrm{msv}} 
\def\gauge{\mathscr{G}}
\def\liegauge{\mathfrak{G}}
\def\hol{\mathrm{hol}}

\def\orbit{\mathcal{O}}
\def\eps{\epsilon}

\def\connections{\mathcal{A}}
\def\bfields{{\tilde {\mathcal{A}} }}
\def\Higgs{\mathcal{H}}

\def\CS{\mathrm{CS}}

\def\norm{||}
\def\Aut{\mathrm{Aut}\,}

\def\LieDer{\mathcal{L}}

\def\Split{{\mathrm{Split}}}

\font\teneurm=eurm10 \font\seveneurm=eurm7 
\font\fiveeurm=eurm5 
\newfam\eurmfam
\textfont\eurmfam=\teneurm \scriptfont\eurmfam=\seveneurm
\scriptscriptfont\eurmfam=\fiveeurm
\def\eurm#1{{\fam\eurmfam\relax#1}}
\font\teneusm=eusm10 \font\seveneusm=eusm7 \font\fiveeusm=eusm5
\newfam\eusmfam
\textfont\eusmfam=\teneusm \scriptfont\eusmfam=\seveneusm
\scriptscriptfont\eusmfam=\fiveeusm

\font\tencmmib=cmmib10 \skewchar\tencmmib='177
\font\sevencmmib=cmmib7 \skewchar\sevencmmib='177
\font\fivecmmib=cmmib5 \skewchar\fivecmmib='177
\newfam\cmmibfam
\textfont\cmmibfam=\tencmmib \scriptfont\cmmibfam=\sevencmmib
\scriptscriptfont\cmmibfam=\fivecmmib

\def\chr{{\mathsf{Chr}}}
\def\curv{\mathrm{curv}}

\def\gerbe{\mathcal{G}}
\def\dbrane{\mathcal{D}}

\def\AG{\mathsf{G}}

\def\Rep{\mathrm{Rep}}
\def\Mod{\mathrm{Mod}}
\def\Cal{\mathrm{Cal}}

\title{Localization for Chern-Simons on Circle Bundles via Loop Groups}

\author{Ryan Mickler}
\affiliation{Department of Mathematics, Northeastern University, 
  Boston, MA 02115 USA}
\emailAdd{mickler.r@husky.neu.edu}
\abstract{We consider Chern-Simons theory on 3-manifold $M$ that is the total space of a circle bundle over a 2d base $\Sigma$. We show that this theory is equivalent to a new 2d TQFT on the base, which we call Caloron BF theory, that can be obtained by an appropriate type of push-forward. This is a gauge theory on a bundle with structure group given by the full affine level $k$ central extension of the loop group $LG$. The space of fields of this 2d theory is naturally symplectic, and this provides a new formulation of a result of Beasley-Witten about the equivariant localization of the Chern-Simons path integral. The main tool that we employ is the Caloron correspondence, originally due to Murray-Garland, that relates the space of gauge fields on $M$ with a certain enlarged space of connections on an equivariant version of the loop space of the $G$-bundle. We show that the symplectic structure that Beasley-Witten found is related to a looped version of the Atiyah-Bott construction in 2-dimensional Yang-Mills theory. We also show that Wilson loops that wrap a single circle fiber are also described very naturally in this framework. }

\begin{document}
\maketitle
\begin{onehalfspace}

\section*{Introduction}

In this paper we will be considering Chern-Simons theory with compact gauge group $G$ on a three-dimensional manifold $M$ which has the structure of a $\BT = S^1$ bundle over a compact base $\Sigma$. Beasley-Witten \cite{Beasley:2005} show, in the case where the $G$-bundle is trivial, that the path integral and Wilson line expected values in this theory on such manifolds are exactly computable via equivariant localization techniques. Their calculations suggest that the theory should reduce to an entirely two-dimensional construction on the base $\Sigma$. We will re-derive their result by showing that we can reformulate Chern-Simons on $M$ as a new 2d topological quantum field theory on the base $\Sigma$, retaining all the degrees of freedom of the 3d theory. We call this theory the Caloron BF theory. This theory is relatively basic to describe, but it has hidden symmetries that are only manifest when describing it using the language of bundle gerbes due to Murray \cite{Murray:1996}. The theory is given as follows. Let $\tilde Q \to \Sigma$ be a $\tilde \AG_k$ bundle, where $\tilde \AG_k$ is the level $k$ central extension of the loop group $\check \AG = \BT \ltimes LG$. Let $\tilde V$ be an adjoint scalar that is constrained to be a section of the adjoint orbit bundle $\tilde Q \times_{\tilde \AG_k} \orbit_{1,0}$, where $\orbit_{1,0} \subset \tilde \Fg_k$ is the adjoint orbit of elements $X$ satisfying $\lllangle X,X\rrrangle = 0 = \lllangle X, c\rrrangle -1$, and let $\tilde L$ be a connection on $\tilde Q$. The Caloron BF action we consider the is the basic 2d BF functional for these fields. \begin{equation}
S^{\mathrm{Cal}}_k = \int_\Sigma \lllangle F_{\tilde L}, \tilde V \rrrangle
\end{equation}
where $\lllangle\,, \rrrangle$ is the non-degenerate invariant pairing on $\tilde \Fg_k$. Our objective is to show that the gauge theory with this action on $\tilde Q$ is classically equivalent to Chern-Simons theory on $M$. Note that there is no coupling constant in this theory; the level $k \in \BZ$ is fixed by the level of the central extension. 

Although the action of this theory is of BF type, the fundamental defining feature is that the adjoint scalar field is constrained to live in an adjoint orbit (defined by a quadratic and a linear constraint) which means that this theory behaves quite differently to a free field theory. For example, we will find that the classical equations of motion for Caloron BF theory reduce to a single equation, rather than a pair of equations which we would expect from usual BF theory. We will see that one of the equations becomes a consequence of the Bianchi identity.

In section \ref{BWsection}, we recall the basics of the Beasley-Witten \cite{Beasley:2005} formulation of the Chern-Simons path integral. Importantly, they assume the triviality of the gauge bundle $P \to M$. Their construction intimately depends on a choice of connection $\kappa$ for the $\BT$ bundle $M \to \Sigma$, which allows for a decomposition of the space of $G$-connections on $M$. Under this decomposition, a certain reduced space of connections inherits a natural symplectic structure. The rigid rotation action of $\BT$ on $M$ lifts to this reduced space and together with the gauge group of $P$ can be made to act in a Hamiltonian way. The square of the moment map for this action recovers a modified form of the Chern-Simons functional. 
We also discuss an extension of this work due to Beasley \cite{Beasley:2013}, where Wilson loop observables that wrap around a single $\BT$ fiber of $M$ can also be included into this symplectic formalism.

In section \ref{sectionGerbesAndTQFT}, we describe the origins of the 2d Caloron BF theory. We use the language of higher bundle gerbes and their holonomy functionals in order to describe certain TQFTs. We recall the constructions of \cite{Carey:2005}, showing that the Chern-Simons amplitude can be interpreted as the holonomy of a certain bundle 2-gerbe. We then \emph{push forward} this 2-gerbe down to the base $\Sigma$, where it becomes a bundle 1-gerbe, and its holonomy defines the action of our 2d Caloron BF theory. 
The main tool used here is the Caloron correspondence, as first described by \cite{Garland:1988}. This constitutes a bijection for each choice of connection $\kappa$ on $M \to \Sigma$, between the space of connections on $P \to M$ and the space of connections and Higgs fields on the equivariant loop space of $P$, that is, the space of maps from $S^1 \to P$ that cover the $S^1$ action on $M$. 

There are also other interesting connections that Chern-Simons theory has with 2d BF theories, as evidenced in the work of Blau-Thompson \cite{Blau:2006}, where it is shown that Chern-Simons theory on non-trivial circle bundles reduces to a type of abelian BF theory. In this way, it is shown that Chern-Simons on such bundles reduces to another 2d TQFT, known as $q$-deformed Yang-Mills theory. This theory has recently been further investigated in \cite{Aganagic:2005}, and it seems likely that there should be a relation between this $q$-deformed theory and the Caloron theory we describe in this work. Perhaps hinting at the well established relationships between the representation theories of the quantum group $U_q(\Fg)$ and of the affine loop algebra $\tilde \Fg_k$.

In section \ref{2dCSsection}, we explore the classical properties of the Caloron BF theory, and show how the critical locus of this theory is bijective with that of Chern-Simons theory. The bundle gerbe constructions of section \ref{sectionGerbesAndTQFT} implicitly guarantee certain properties of the Caloron BF action, for example, independence of the choice of lifting data. Here we will provide some explicit checks of this invariance, using the same techniques as is done for regular Chern-Simons theory. Ultimately, we show that the Caloron correspondence preserves the dynamical structure of the spaces of fields.

In section \ref{BWRevisited}, we return to the motivating subject of this work, which is to show that the symplectic structure discovered by Beasley-Witten for Chern-Simons theory is related to a version the Atiyah-Bott symplectic form for the Caloron BF theory. Recall that in the Beasley-Witten formulation, the action of rigid global $\BT$ rotations of $M$ is lifted to the space of reduced connections. We find that in the Caloron BF theory we must lift the action of \emph{local} $\BT$ transformations of $M$, i.e. gauge transformations of the $\BT$-bundle, to the space of reduced connections. In this way, the specific dependence on $\kappa$ is reduced to a dependence only on the gauge-equivalence class. This is the fundamental difference between the Caloron formulation and that of Beasley-Witten. In \cite{Imbimbo:2014}, an alternative approach to (supersymmetric) Chern-Simons on circle bundles is presented, where the background connection $\kappa$ is also considered as a dynamical field. 

\subsection*{Acknowledgements}
The author is grateful to Chris Beasley, Michael Murray, Mathai Varghese, Ray Vozzo, and Pedram Hekmati for stimulating discussions and for comments on drafts of this paper. The author also gratefully acknowledges support from the Simons Center for Geometry and Physics, Stony Brook University where a portion of the research for this paper was performed. This material is based upon work supported by the National Science Foundation under Grant Number PHY-1620637.  Any opinions, findings, and conclusions or recommendations expressed in this material are those of the author and do not necessarily reflect the views of the National Science Foundation.

\section{The Beasley-Witten Construction} \label{BWsection}

Here we will review the constructions of Beasley-Witten \cite{Beasley:2005, Beasley:2013}, which show that the Chern-Simons path integral is symplectic of `norm-squared' type. A classical result of Atiyah-Bott \cite{Atiyah:1983} shows that the 2d Yang-Mills path integral can be formulated as a symplectic integral. That is, the space of connections on a $G$-bundle $P$ over a Riemann surface $\Sigma$ can be equipped with the canonical sympelctic form 
\begin{equation}\label{ABSymplecticForm}
\Omega= \int_\Sigma \langle \delta A, \delta A \rangle.
\end{equation}
The group of gauge transformations of $P$ act in a Hamiltonian way on this space, with moment map given by the curvature $\mu = F_A$. Once we equip $M$ with a volume form, the square of this moment map is precisely the Yang-Mills action $S(A) = \int_\Sigma || F_A ||^2$. The Yang-Mills path integral can then be evaluated using the techniques of equivariant localization, which we discuss in section \ref{SymplecticStructure}. By comparison, for Chern-Simons theory on a 3-manifold $M$, there is no such known symplectic structure. However, Beasley-Witten show that if we restrict to those manifolds that admit a (locally) free $\BT$ action, then such a structure does exist, but not canonically.

Let $M$ be a three manifold that has the structure of a principal $\BT$ bundle over a compact surface $\Sigma$, and let $P$ be a trivial $G$ bundle on $M$. The space of connections $\connections_P$ on $P$ is equivalent to the space of $\Fg$-valued 1-forms on $M$. The Chern-Simons 3-form for a connection $A \in \forms^1(M,\Fg)$ is given by
\begin{equation}\label{CSfunctional}
\cs(A) := \frac{1}{8\pi^2}  \langle A, F_A - \tfrac{1}{6}[A,A]\rangle
\end{equation}
This form satisfies $d \cs(A) = \tfrac{1}{8\pi^2}\langle F_A,F_A\rangle$, and its integral over $M$ constitutes the action of Chern-Simons theory. 
\begin{equation}
CS_k(A) := k \int_M \cs(A)
\end{equation}

We now make a choice of a connection form $\kappa \in \forms^1(M, \BR)$ on $M$ for the $\BT$ action. This form allows us to project forms on $M$ into their horizontal component using the operator $\bar P_\kappa =1- \kappa \iota_R$, where $R$ is the vector field generating the $\BT$ action. We denote by $a = \bar P_\kappa A$ the horizontal component of the connection $A$, and the space of all such components as $\bar \connections_\kappa$. Next, we introduce a 2-form on $\bar \connections_\kappa$ given by 
\begin{equation}
\Omega_{\kappa} = \tfrac{1}{2}\int_M \kappa \langle \delta a \wedge \delta a \rangle
\end{equation}
This can easily be shown to be symplectic. The group of gauge transformations $\gauge_P$ acts on the space of connections in the standard way. This action descends to $\bar \connections_\kappa$ via
\begin{equation}
\label{barAaction}
a \mapsto g^{-1} a g + g^{-1}(d - \kappa \LieDer_R) g, \qquad g \in G^M
\end{equation}
The infinitesimal action of the Lie algebra $\Fg^M = Lie(\gauge_P)$ is $\iota_\xi \delta a =  D_a \xi := (d+\ad_a - \kappa \LieDer_R)\xi$. 
If we now assume that the $\BT$ bundle is topologically non-trivial, then the top degree form $\kappa \wedge d\kappa$ is non-vanishing, and hence defines a volume form on $M$. In this way, $\kappa$ is also known as a contact structure on $M$.

With this contact strucutre, the algebra $\Fg^M$ possesses an invariant, non-degenerate inner product given by
\begin{equation}
\llangle \xi_1,\xi_2 \rrangle = \int_M \kappa \wedge d\kappa \langle \xi_1, \xi_2 \rangle
\end{equation}
We extend the group of gauge transformations of $P$ to include the rigid $\BT$ action on $M$. The infinitesimal action is $\iota_u \delta a = u \LieDer_R a$. These two actions combine to give an action of the semi-direct product $\liegauge' := \Ft \oplus \Fg^M$ by $\iota_{(u,\xi)} \delta a = D_a \xi + u \LieDer_R a$.
This $\BT$ action on the space of connections is only possible in the case where the $G$-bundle is trivial, since connections can be identified with $\forms^1(M,\Fg)$. If the bundle $G$ is non-trivial, (e.g. this can happen when $G = SO(3)$), there is no such lift of the $\BT$ action to connections. In the case where the bundle is non-trivial, the constructions in the later sections of the paper become essential, and we will find that we need to lift everything to the loop space in order to find such a $\BT$ action. However, if we continue to assume that the bundle is trivial, we find
\begin{equation}
\iota_{(u, \xi)}\Omega_\kappa = \delta \left( -\int_M \kappa \langle \xi, D_a a\rangle+ u \int_M \kappa \langle \LieDer_R a \^ a\rangle \right)=: \delta v_a(u,\xi)
\end{equation}
This map $v : \liegauge' \to C^{\infty}(\bar \connections_\kappa)$ is not a (co-)moment map however, since is it not a Poisson map (i.e. equivariant). There exists a cocycle \begin{equation}
k_0( (u_1,\xi_1), (u_2,\xi_2) ) := v([(u_1,\xi_1), (u_2,\xi_2) ]) - \{ v(u_1,\xi_1), v(u_2,\xi_2) \}
\end{equation} given by
\begin{equation}\label{bwextension} k_0(\xi_1,\xi_2) = \llangle \LieDer_R \xi_1, \xi_2 \rrangle
\end{equation}
This provides a central extension $ \Ft_c \to \tilde \liegauge' \to \liegauge' $. The similarity with the central extension of the loop group was noted by the authors, which motivates a lot of the constructions in the later sections of this work. This algebra also exhibits an non-degenerate, invariant inner product
\begin{equation} \lllangle (u_1,\xi_1,c_1),(u_2,\xi_2,c_2) \rrrangle =\llangle \xi_1, \xi_2\rrangle - u_1c_2 - u_2 c_1
\end{equation}
If one extends the action of central component of $\tilde \liegauge'$ to $\bar\connections_\kappa$ trivially, i.e. $\iota_c \delta a = 0$, then the surprising result is
\begin{theorem}[Beasley-Witten]
The above action of $\tilde \liegauge'$ on $(\bar\connections_\kappa, \Omega_\kappa)$ is Hamiltonian with moment map $\mu : \bar\connections_\kappa \to (\tilde\liegauge')^\vee$
\begin{equation}\label{BWMOMENT}
\mu = -\left( 1 , \frac{\kappa \^F_a - d\kappa \^a}{\kappa \^ d\kappa}, \frac{1}{2} \int_M\kappa \langle a, \LieDer_R a\rangle \right) 
 \end{equation}
Furthermore, the square of this moment map is equal to the following contact Chern-Simons action $\CS'(a,\kappa)$ on $\bar \connections_\kappa$.
\begin{equation}
\lllangle \mu,\mu \rrrangle  = CS'(a,\kappa) := CS(a) - \llangle f ,f \rrangle 
\end{equation}
where $f \in \forms^0(M,\Fg)$ is defined by $f \kappa\wedge d\kappa =\kappa\wedge F_a$.
\end{theorem}
Thus the Beasley-Witten result shows that this contact Chern-Simons functional can be realized as the square of a moment map. They also show that this contact Chern-Simons action is gauge equivalent to the regular Chern-Simons action. In the sense that,
\begin{equation}
\int_{\connections_P} [DA] \exp\left( \Ri CS_k(A) \right)  = \int_{\bar \connections_\kappa} [Da] \exp\left( \Ri  CS_k'(a,\kappa) \right) = \int_{\bar \connections_\kappa} \exp\left( \Ri k \lllangle \mu,\mu \rrrangle + \Omega_\kappa \right).
\end{equation}
Using the tools of non-abelian localization, the integral can be evaluated in this form explicitly.

It was observed in \cite{Beasley:2013} that Wilson loop operators around a single fiber of $M\to \Sigma$ also have a symplectic interpretation. Recall the classical definition of a Wilson loop operator. Let $C$ be a loop in $M$, and $R$ a representation of $G$. Then we define
\begin{equation}
W_R(C,A) = \Tr_R\, \mathrm{P}\!\exp\left( -\oint_C A\right)
\end{equation}
That is, the Wilson loop $W_R(C,A)$ computes the trace of the holonomy of the connection $A$ in the representation $R$ around the loop $C$. It has been known for some time that this operator can also be interpreted as the trace of an auxiliary quantum system attached to $C$, coupled to the background connection $A$, c.f. \cite{Witten:1989, Elitzur:1989}. Suppose $R$ is an irreducible representation with highest weight $\alpha \in \Fh^\vee$, which is dual to $\alpha^\vee \in \Fh$. Let $\orbit_\alpha \subset \Fh$ be the adjoint orbit of $\alpha^\vee$. Then consider a 1d sigma model of maps $U: C\to \orbit_\alpha$ coupled to $A$, with action
\begin{equation} \cs_\alpha(U,A) = \oint_C \langle \alpha^\vee , g^{-1} d_A g \rangle 
\end{equation}
where $U(\theta) = g(\theta) \alpha^\vee  g(\theta)^{-1} $ and $d_A g = d g + A|_C g$. It can then be shown (c.f. \cite{Beasley:2013}) using path-integral techniques that the Wilson loop operator can be expressed as the following sigma model action
\begin{equation}
W_R(C,A) = \int_{Map(C, \orbit_\alpha)} \mathcal{D}U \exp\left(\Ri \,\cs_\alpha(U,A) \right) 
\end{equation}
It was through this expression that Beasley was able to show that this operator was compatible with the symplectic structure; even with inclusion of this operator inside the path-integral it still can be expressed in the `normed-squared' form.

\section{Bundle Gerbes and The Caloron Correspondence} \label{sectionGerbesAndTQFT}

We now embark on our journey to reformulate the previous results of Beasley-Witten in terms of a new 2d TQFT on the base of the fibration $\Sigma$.
We consider the data of a (possibly non-trivial) principal $G$ bundle $p: P\to M$ over a circle bundle $\pi : M\to \Sigma$. The desired reformulation is achieved by pushing forward a geometric object known as the Chern-Simons bundle 2-gerbe along $\pi$ to get a 1-gerbe on $\Sigma$. This 1-gerbe will describe our 2d TQFT. For the reader who is uninterested in the origins of the Caloron BF theory, they should be able to skip ahead to section \ref{2dCSsection}, where we discuss the classical properties of the theory using traditional techniques.

We will begin assuming that the geometry $P \to M^3\to \Sigma^2$ equivariantly embeds into a bulk geometry of one higher dimension $P \to Y^4\to B^3$, as in the following diagram
\begin{equation}
\begin{gathered}
\xymatrix{ 
P \ar[d] \ar[r]  & P \ar[d]_p \\
M \ar[d] \ar@{->}[r] & Y \ar[d]_\pi \\
\Sigma \ar@{->}[r]& B
}
\end{gathered}
\end{equation}
where we use the notation $P$ for both the bundle on the bulk $Y$ and for its restriction $P|_M$.
The usual case is when $M \hookrightarrow Y$ is the inclusion of the boundary $\partial Y = M$, however unless specified we will consider an arbitrary inclusion. In general, there will be obstructions to finding such a bounding geometry (see \ref{Obstruction}), however, we use it solely to motivate certain constructions on $\Sigma$ that exist even in the case when a bounding geometry does not. 

The first step toward describing the 2d theory on $\Sigma$ is to recall how Chern-Simons theory on $M$ can be expressed using the language of higher abelian gerbes. We will recall some of the formalism of bundle gerbes and how they capture the geometry of higher abelian gauge fields. Then we summarize some key results of Murray-Stevenson \cite{Murray:2003}, and Murray-Vozzo \cite{Murray:2010}, that provide a relation between the Chern-Simons bundle 2-gerbe on the 4d bulk $Y$, which is a circle bundle over a base $B$, and the lifting bundle gerbe associated to the equivariant loop space over $B$.

We will show that the geometry of this bundle gerbe on the 3d bulk $B$ describes a 2d topological gauge theory in the boundary $\Sigma$, that is equivalent to Chern-Simons theory upstairs on $M$.

\subsection{Gerbes and Differential Characters}

In Dijkgraaf-Witten \cite{Dijkgraaf:1990} it was noticed that the Chern-Simons level $k$ corresponds to a class $[k] \in H^4(BG,\BZ) $, and that the Chern-Simons amplitude $\exp\left( 2\pi \Ri k\,CS(A)\right)$ provides a \emph{differential refinement} of the class $[k]$ in the differential cohomology of Cheeger-Simons \cite{Cheeger:1985}. Essentially, it was noted that if the 3-manifold $M$ bounds a bulk $Y$ on which the bundle $P$ extends, then the formula
\begin{equation}
CS(A) = k\int_Y p_1(F_A)\,\, \mod \BZ
\end{equation}
where $ p_1(F_A)= \tfrac{1}{8\pi^2}\langle F_A, F_A \rangle$ is the de Rham representative of the 1st Pontryagin class,
can be used as a definition of the Chern-Simons action on $M$ even when the bundle $P$ is non-trivial.

We recall that the group of Cheeger-Simons differential characters are defined by $\check H^{n}(X) = \{(h, \omega)\} \subset \Hom(Z_{n-1}^{smth}(X), \BT) \times \Omega^{n}_{\BZ}(X,\BR)$, satisfying
\begin{equation}
h(\partial C) = \exp 2\pi\Ri \int_C \omega, \text{ for all smooth cycles } C \in Z_{n}^{smth}(X).
\end{equation}
These classes provide an extension to higher degrees of the notion the holonomy $h$ and curvature $\omega$ of a principal circle bundle with connection. In particular, for a circle bundle with connection $\mathcal{L} = (L, \kappa) $ its differential character is $\chr(\mathcal{L}) := (\hol_{\mathcal{L}}(\kappa), d\kappa) \in \check H^2(X)$. Now notice that if $C$ is a curve in $Y$ and $s$ a trivialization of $\mathcal{L}|_C$. The action
\begin{equation} S(C,\kappa) := \log \hol_{\mathcal{L}}(C)= 2 \pi \Ri \int_C s^* \kappa \end{equation}
is the contribution to the Lagrangian of a charged point particle moving in the background gauge field $\kappa$. This action is independent modulo $\BZ$ on the choice of section $s$. In an analogous way, Dijkgraaf-Witten showed that for a 3-manifold that bounds a 4d bulk $Y$, we have
\begin{equation} \CS_k(A) := ( \exp (2\pi \Ri \, CS_k(A)), k \,p_1(F_A)) \in \check H^4(Y)\end{equation}
is a differential character that refines $f^*[k] \in H^4(Y,\BZ)$, where $f : Y \to BG$ is the classifying map of the bundle $P \to Y$.

We will use the technology of bundle $n$-gerbes and their holonomy functors, as they provide a convenient model for cocycles in differential cohomology.  We refer to the works \cite{Murray:1996,Murray:2003,Carey:2005} for the basic material on bundle gerbes, and we use the modern notation of \cite{Fuchs:2010}.  We denote the trivial bundle $n$-gerbe with curving $\varrho \in \forms^{n-1}_{cl}(Y)$ as $\mathcal{I}_\varrho$. Given a bundle $n$-gerbe $\gerbe$ over $Y$, and a trivialization $t : \gerbe|_\Sigma \to \mathcal{I}_\varrho$ over an $n-1$-dimensional sub-manifold $X \subset Y$, the holonomy is defined as
\begin{equation} \hol_\gerbe(X,t) := \exp 2\pi \Ri \int_X \varrho \end{equation}
Importantly, this expression is independent of the choice of trivialization c.f. \cite{Carey:2005}. We consider the functor $\chr :n\mathrm{Grb}(Y) \to \check H^{n+2}(Y)$,
which sends the $n$-gerbe $\gerbe$ to the differential character $(\hol_{\gerbe},\curv_\gerbe) \in \check H^{n+2}(Y)$. Effectively, this map is sending a (differential) cocycle to its class in (differential) cohomology, c.f \cite{Murray:1996}.

The basic philosophy propagated in this paper is as follows: \emph{The (log) holonomy of certain bundle $n$-gerbes yield topological terms in action functionals of $n\!+\!1$ dimensional quantum field theories. }

This philosophy was manifest in Carey et al \cite{Carey:2005} where the Dijkgraaf-Witten result was refined to the level of gerbes. It was shown that the Chern-Simons action on $M$ can be expressed as the holonomy of a bundle 2-gerbe. 
\begin{theorem}[\cite{Carey:2005}]
There exists a bundle 2-gerbe $\CSg_k(A)$ (with connection and 2-curving determined by the connection $A$) over the 4d bulk $Y$, called the \emph{Chern-Simons 2-gerbe}, such that
\begin{equation}
\chr(\CSg_k(A)) = \CS_k (A)
\end{equation}
\end{theorem}
The structure of this 2-gerbe intrinsically captures the interesting transformation properties of the Chern-Simons action. Another such example is the topological term of the Wess-Zumino-Witten model, which can also be expressed as the holonomy of a bundle gerbe on the group manifold $G$ that represents the fundamental class in $H^3(G,\BZ)$.

\subsection{The Push Forward}

Recall our geometric set up:

\begin{equation}
\begin{gathered}
\xymatrix{ 
P \ar[d] \ar[r]  & P \ar[d]_p \\
M \ar[d] \ar@{->}[r] & Y \ar[d]_\pi \\
\Sigma \ar@{->}[r]& B
}
\end{gathered}
\end{equation}

Our objects of interest are the Chern-Simons gerbe $\CSg_k(A) \in 2\mathrm{Grb}(Y)$, and its differential character $\CS_k(A) \in \check H^4(Y)$. In order to reformulate Chern-Simons theory as something manifestly 2-dimensional, we will \emph{push-forward} this data along the bundle $\pi$. For this purpose, we recall the definition of fiber integration in differential cohomology.
\begin{proposition}[\cite{Hopkins:2005}]
Let $f : M \to N$ be a fiber bundle with dimension $k$ fibers and $[\theta] \in H^k(M,\BZ)$ an orientation class, then there is a functorial map
\begin{equation} \check f_*^{[\theta]} : \check H^n(M) \to \check H^{n-k}(N)\end{equation}
with the following defining properties: for $(h,\omega) \in \check H^n(M)$, let $(h',\omega') = \check f_*^{[\theta]} (h,\omega) \in \check H^{n-k}(N)$, then
\begin{equation}
\omega' = \int_{M/N}^{[\theta]} \omega, \qquad  h'( C ) = h(f^{-1} (C))
\end{equation}
for $C \in Z_{n-k}(N)$.
\end{proposition}
We are interested in refining this fiber integration operation to the level of $n$-gerbes. However, we only know of a local construction of this map at the level of cocycles in Deligne cohomology \cite{Gomi:2000}, which is another model for differential cochains. Nevertheless, we will find that we can still make a definition of what the push-forward of the Chern-Simons 2-gerbe \emph{should} be. For our situation, the circle bundle $\pi: Y \to B$ is canonically oriented, and a connection $\kappa$ yields an orientation form. Thus, ideally we would like to define a map $\check\pi^{\kappa}_* : 2\mathrm{Grb}(Y) \to  1\mathrm{Grb}(B)$, such that the following diagram is commutative,
\begin{equation}
\begin{gathered}
\xymatrix{ 
2\mathrm{Grb}(Y)  \ar[r]^{\chr} \ar@{.>}[d]^{\check\pi^{\kappa}_*=? }
 & \check H^{4}(Y) \ar[d]^{\check\pi^{[\kappa]}_* }\\ 
 1\mathrm{Grb}(B)  \ar[r]^{\chr} 
 & \check H^{3}(B) \\ 
}
\end{gathered}
\end{equation}
However, for our purposes we only need to complete the following diagram
\begin{equation}
\begin{gathered}
\xymatrix{ 
\CSg_k(A)  \ar[r]^{} \ar@{.>}[d]^{}
 & \CS_k(A) \ar[d]^{ }\\ 
?  \ar[r]^{} 
 & \check\pi^{[\kappa]}_* \CS_k(A) \\ 
}
\end{gathered}
\end{equation}

The object that will make this diagram commute is given by what is known as a lifting bundle gerbe, and its $log$ holonomy will define the action functional of what should be the push forward of Chern-Simons theory on $M$.

\subsection{Lifting Bundle Gerbes}

\def\curv{\mathrm{curv}}
Given a principal $K$-bundle $\pi : P \to B$ and a central extension of the structure group $\BT \to \tilde K\to K$, there is an associated lifting bundle gerbe, denoted $\mathbb{L}_{P}^{\tilde K}$. For a full definition see \cite{Murray:1996}, we also give a review of this material in the appendix. This gerbe is a geometric representation of the obstruction to lifting the structure group of $P$ to $\tilde K$, which lies in $H^3(B,\BZ)$. A lift $\tilde P \to P$ gives a topological trivialization of the gerbe. The most familiar example is that of lifting the structure group of $TX\to X$ from $SO(n)$ to $\mathrm{Spin}_{\BC}(n)$, whose obstruction is $W_3(X) \in H^3(X,\BZ)$. The associated $\mathrm{Spin}_{\BC}$ bundle gerbe was considered in \cite{Murray:2004}.

The following construction is due to Brylinski \cite{Brylinski:2008} in the setting of Deligne cohomology, and then in the case of loop groups by Gomi \cite{Gomi:2003}, and further extended by Vozzo \cite{Vozzo:2009}. Given a connection $A$ for $P$, and a bundle splitting $s: P \times_K \tilde \Fk \to P \times_K \Ft$ (i.e. a map that is identity on the central factor), we can put a connective structure on $\mathbb{L}_{P}^{\tilde K}$, which we denote $\mathbb{L}_{P}^{\tilde K}(A,s)$. 

\begin{proposition}[Brylinski, Gomi, Vozzo]\label{holbryl}
Let $p: \tilde P \to P$ be a lift to a $\tilde K$ bundle, and $\tilde A$ be a lift of the connection $A$. This data provides a trivialization
\begin{equation} \tau_{\tilde P,\tilde A} : \mathbb{L}_P^{\tilde K}(A,s)  \to \mathcal{I}_{\varrho} \end{equation}
where $\varrho \in \forms^2(B)$ is defined by $ (p\circ \pi)^* \varrho = s(F_{\tilde A})$.
\end{proposition}
Given a trivialization as above, we can easily compute the holonomy of the lifting gerbe.

\begin{lemma}[Gomi, Murray]\label{holliftgerbe}
The holonomy of the lifting bundle gerbe over a surface $\Sigma$ with a chosen lift $(\tilde P,\tilde A)$ is given by 
\begin{equation} \hol_{\mathbb{L}_P^{\tilde K}(A,s)}(\Sigma, \tilde P,\tilde A) = \exp 2\pi\Ri \int_\Sigma s(F_{\tilde A}).\end{equation}
This holonomy is independent of the choice of lifting data $(\tilde P, \tilde A)$.
\end{lemma}

This independence of the choice of lifting data will be crucial for our later discussion of the 2d reformulation of Chern-Simons theory.

\subsection{The Lifting Bundle Gerbe of the Equivariant Loop Space}
We want to construct a gerbe on the 3d bulk $B$ out of the data $P \stackrel{p}{\rightarrow} Y \stackrel{\pi}{\rightarrow} B$. We do this by considering an equivariant version of the loop space $LP$, as was originally considered by Bergman-Varadarajan \cite{Bergman:2005} for describing Kaluza-Klein reductions.
\begin{definition}
The \emph{equivariant loop space} $Q_P \subset LP$ is the set of maps $f: \BT \to P$, such that $p \circ f :  \BT \to M$ is equivariant with respect to the $\BT$ action on $Y$.
\end{definition}
Now $Q_P$ possesses a $\check \AG = \BT \ltimes LG$ action given by 
\begin{equation}
\label{qpaction} f(\phi,\gamma)(\theta) := (\rho_{-\phi}(f \gamma))(\theta) = f(\theta+\phi)\gamma(\theta+\phi) .
\end{equation}
It can be easily shown that action is free, and thus $Q_P \in \mathrm{Bun}_{\LGS}(B)$. Thus on this larger bundle, both the $G$ action and $\BT$ actions are unified as an action of $\check \AG$.  Furthermore, the evaluation map $E$ at $0 \in \BT$ provides a map $Q_P \to P$. This map is invariant under the action of the based loops $\Omega G \subset \check \AG$ on $Q_P$, and thus $\Omega P \to Q_P \to P$ is a principal bundle. Since the subgroup of constant loops $G \subset LG$ is contained in the normalizer of $\Omega G$, i.e. $G\subset N_{\Omega G}(LG)$, the action of $G$ on $Q_P$ descends to the action of $G$ on $P$. All of these structures fit into following commuting diagram
\begin{equation}\label{qassocdiag}
\begin{gathered}
\xymatrix{ 
Q\times  \check \AG \ar[d]^{\eps} \ar[r] & Q \ar[d]^E\\
Q\times  (\BT \times G) \ar[d]_{G} \ar[r] & P \ar[d]^p\\
Q \times\BT \ar[d]_{\BT} \ar[r]& M \ar[d]^{\pi}  \\
Q \ar[r]&  \Sigma \\
}
\end{gathered}
\end{equation}

Here, the new vertical map is $\eps: \check \AG \to \BT \times G$, $\eps(\theta, \gamma) = (\theta, \gamma(\theta))$. The horizontal maps are given by the associated bundle construction. In order, they are $(q,\check \gamma) \mapsto q\check \gamma \in Q$, $(q,\theta, g) \mapsto q(\theta)g \in P$, $(q, \theta) \to (p \circ q)(\theta) \in M$, $q \mapsto (\pi \circ p \circ q) \in \Sigma$.

\begin{proposition}[Bergman-Varadarajan]
There is an equivalence between the category of $\check \AG$-bundles $Q$ over $\Sigma$, and the category of $G$-bundles $P$ over circle bundles $M$ over $\Sigma$.
\end{proposition}
The following corollary is an analogue of the obstruction computed by Dijkgraaf-Witten \cite{Dijkgraaf:1990} in the case of finding an extension of the $G$-bundle over $M$ to some bounding 4-manifold  $Y$. 
\begin{corollary}\label{Obstruction}
The obstruction to finding a bounding geometry $P \to Y\to B$, is the same as the obstruction of finding a bounding $\check \AG$-bundle over $B$.
\end{corollary}

For an alternate perspective, $Q_P$ can equivalently be described fiber-wise with fiber above a point $b \in B$ given by the data $(t, s)$, where $t$ is a choice of equivariant identification $t:  M_b \stackrel{\sim}\rightarrow S^1  $, and $s$ is a section $s : M_b \to P|_{M_b}$. The action of $\check \AG$ is then given by
\begin{equation}
(t,s)(\phi,\gamma) = (\rho_\phi^{-1} t, s(\gamma \circ t) )
\end{equation}
then we check
\begin{eqnarray}
( (t,s)(\phi_1,\gamma_1) )(\phi_2, \gamma_2) &=& (\rho_{\phi_2}^{-1} \rho_{\phi_1}^{-1} t, s(\gamma_1 \circ t) (\gamma_2 \circ \rho_{\phi_1}^{-1} t) ) \\
& =& (t,s)(\phi_1+\phi_2, \gamma_1 \rho_{\phi_1}\gamma_2)
\end{eqnarray}

From this perspective, we can see that equivariant loop space resembles a push-forward construction.

\subsection{Connective Structures}

Here we will show that a we can extend the above correspondence to include connective structures, i.e. connections. On the 3d side, as in the Beasley-Witten construction, we have a connection $A$ on $P \to M$, and a connection $\kappa$ on $M \to \Sigma$. On the loop space side we want to consider $\check \AG$-connections on $Q$, which we denote $L = (\lambda, \Lambda) \in \forms^1(Q, \Ft \oplus \flg)$, as well as a scalar field, known as a Higgs field.
\begin{definition}[\cite{Murray:2010}]
Let $Q$ be a  $\LGS$-bundle over $B$. A \emph{Higgs Field} for $Q$ is a section $V$ of the adjoint orbit bundle $Q \otimes_{\check \AG} \orbit_1$, where $\orbit_{1} \subset \check \Fg$ is the adjoint orbit of elements $X$ of the form $X = (1,-\Phi)$,
\end{definition}

We denote the space of Higgs fields on $Q$ by $\Higgs_Q$. We should think of a connection $L$ as belonging to the horizontal directions of the geometry over $\Sigma$, and the Higgs field $V$ is a family of connections on the fibers $F \cong G\times S^1$.
We also have several equivalent characterizations of Higgs fields, given in \cite{Hekmati:2012}, which we will occasionally use. Firstly, if we write $V \equiv V_\Phi := (1,-\Phi)$ in components, then $\Phi$ satisfies $\Phi(q(\gamma,\phi)) = \rho_{\phi}^{-1}(\Ad_{\gamma^{-1}} \Phi(q) + \gamma^{-1} \partial \gamma )$. Thus, importantly, we can see that the space of Higgs fields is an affine space modeled on sections of the adjoint bundle $Q \times_{\Ad} L\Fg \subset Q \times_{\Ad} \check \Fg$. Furthermore, for $q \in Q$, the map $q \to d - \Phi(q) d\theta$ defines a $\check \AG$-equivariant family (over $Q$) of connections on the trivial $G$ bundle over the circle, where we identify $\check \AG \subset \Aut(S^1 \times G)$. It is in this way we see that a Higgs field really is a family of connections on the fibers of $Q_P$. Lastly, for a Higgs field $V$ consider the map $\beta(\cdot,V) \in \Gamma( Q \times_{\Ad} \Split(\flg \to \check \Fg))$, where $\beta$ is given by 
\begin{equation}\label{betadef}
\beta: \check \Fg \times \check \Fg \to \check \Fg, \qquad \beta((t,X), (s,Y)) = (0,sX-tY).
\end{equation}
It is easy to check that this map is $\check \AG$ equivariant. To show that this induces a spitting, we check $ \beta((0,\eta),(1,-\Phi)) = \eta$.

\subsection{The Caloron Correspondence}\label{Caloronsection}

We can now state the relationship between connections on $P$ and connections on the equivariant loop space. This correspondence was first used by Garland-Murray \cite{Garland:1988} where they constructed an equivalence between periodic $G$-instantons on $\BR^3 \times S^1$ and monopoles on $\BR^3$ for the based loop group $\Omega G$.
\begin{proposition}[Caloron Correspondence, \cite{Murray:2010}]
There is a bijection between the set of pairs of connections $(\kappa,A) \in \connections_M \times \connections_P$ on $P \to M \to \Sigma$, and the set of pairs of connections and Higgs fields $(L, V) \in \connections_Q \times \Higgs_Q$ on $Q_P \to \Sigma$.
\end{proposition}
This expresses the idea that to build a connection on $P$ from data on $Q$, we need a connection for $Q$, as well as a family over $Q$ of connections on the fibers.

Let us briefly recall the construction in \cite{Murray:2010}. Let $j : P \to M$ be the $G$-bundle, and $\pi : M \to \Sigma$ be the circle bundle, with generating vector field $R$. Let $q: \BT \to P$ be an equivariant loop. With our conventions, the $\BT$ component of $\LGS$ acts on $Q_P$ inducing a rotation vector field $\tilde R$, this induces rotation on $M$. So if we let $\bar q = j \circ q : \BT \to M$, we have $\bar q_* \tilde R = R$. Begin with choice of $G$-connection $A \in \forms^1(P,\Fg)$ on $P$, and let $q \in Q_P$ be an equivariant loop in $P$. Define $\tilde A \in \forms^1(Q_P, \flg)$, by $(\tilde A)_q(\xi) = (q^*A)(\xi)$, where $q^*A \in \Gamma(\BT, q^*(T^*P)\otimes \Fg)$, and $\xi \in T_q Q_P = \Gamma(\BT, q^*(TP))$, thus $(q^*A) (\xi) \in \Gamma^0(\BT,\Fg) = L\Fg$, i.e. $(\tilde A)_q(\xi)_\theta = A_{q(\theta)}(\xi_\theta)$. From the equivariance of $A$, can easily see that $r^*_\gamma \tilde A =\Ad_{\gamma^{-1}} \tilde A$. 
Furthermore, under rotations we have $ r^*_\phi(\tilde A)_q (\xi)_\theta  =  (q^* A)((r_{\phi})_* \xi)_\theta =A_{q(\theta)}(\xi_{\theta+\phi}) $. This can equivalently be expressed as $\LieDer_{\tilde R} \tilde A = \partial \tilde A $.
 Thus we conclude that 
 \begin{equation}\label{liftedA}
r_{(\phi,\gamma)}^* \tilde A = \Ad_{(\phi,\gamma)^{-1}} \tilde A,\qquad \iota_{(0,\xi)} \tilde A = \xi 
\end{equation}
We call such an object a \emph{looped connection}.

Now, given a $\BT$-connection $\kappa \in \forms^1(M, \Ft)$, define $\tilde \kappa  \in \forms^1(Q_P,\Ft)$ by $\tilde \kappa_q = \bar q^* \kappa$, so that $\iota_{\tilde R} \tilde \kappa = \iota_R \kappa =1$. Now, consider the projection operator $\bar P_{\tilde \kappa} = 1 -  \tilde \kappa \iota_{\tilde R}$. Define
\begin{equation} L = (\lambda, \Lambda) := (\tilde \kappa, \bar P_{\tilde \kappa}\tilde A) \in \forms^1(Q_P,\flgs), \quad V := (1,-\iota_{\tilde R} \tilde A) \in \forms^0(Q_P,\check \Fg) \end{equation}

\begin{lemma}[\cite{Murray:2010}]
The pair $(L,V)$ is a $\LGS$ connection and a Higgs field on $Q_P$. Specifically, this connection is the one obtained by splitting tangent spaces to $Q_P$ point-wise using the data $(A,\kappa)$.
\end{lemma}

To build a connection $A$ out of the data on the Caloron side, we invert the above construction. Given $(L,V)$ as before, we define
\begin{equation} (0,\tilde A) := \beta(L,V) = (0, \Lambda + \lambda \Phi) \in  \forms^1(Q_P,\check \Fg).\end{equation}
By the equivariance of $\beta$, we see that $\tilde A$ is a looped connection, and corresponds uniquely to a connection $A$ on $P$. 
It is also easy to see that the component $\lambda \in \forms^1(Q, \Ft)$, is basic for the $ LG$ action on $Q$, and thus descends to a form $\kappa\in \forms^1(M, \Ft)$, which is a connection for the $\BT$ action on $M$.
\begin{lemma}[\cite{Murray:2010}]
These constructions are mutually inverse (up to equivalence).
\end{lemma}

We denote this correspondence by
\begin{equation}
\Cal : \connections_M \times \connections_P \to \connections_Q \times \Higgs_Q,\qquad  (\kappa, A) \mapsto (L,V).
\end{equation}
We called the looped connection form $\tilde A = \beta(L,V)$ the \emph{Caloron connection}.
Our goal in the rest of this paper will be to use this correspondence to translate the geometry of the Chern-Simons 2-gerbe on the left hand side, to the geometry of the loop-space Lifting bundle gerbe on the right hand side.

\subsection{The Push-Forward}

We can now return to understanding how the geometry of the lifting bundle gerbe for $Q \to \Sigma$, is related to the Chern-Simons bundle 2-gerbe for $P \to Y$. Recall that we equip the equivariant loop space $\pi_Q: Q_P \to B$, with a connection $L$ and Higgs field $V = (1,-\Phi)$.  We then define $\tilde V = (1,-\Phi, \tfrac{k}{2} \llangle \Phi,\Phi\rrangle)$, which is a section of the bundle $Q \times_{\check \AG} \tilde \Fg_k$, where $\tilde \Fg_k$ is the level $k$ central extension. We can see that $\tilde V$ is valued in the adjoint orbit of the element $(1,0,0) \in \tilde \Fg_k = \Ft_R \oplus L\Fg \oplus \Ft_c$. We can use this element to construct a bundle splitting for $Q$ defined by $s_{V}(X) := \lllangle X ,\tilde V\rrrangle$. That is
\begin{equation}
s_V : Q \times_\Ad \tilde \Fg_k \to Q \times_\Ad \Ft_c \cong B \times \Ft_c
\end{equation}
is identity on the central factor.

Recall the lifting bundle gerbe $\mathbb{L}_{Q}^{\tilde \AG_k}(L,V)$, for the level $k$ central extension $\BT \to \tilde  \AG_k \to \check  \AG$, can be equipped with the Brylinski-Gomi connective structure (described in the appendix), using the data of a connection $L$ and a bundle splitting $s_V$. For simplicity, we denote this gerbe with connection and curving by $\mathbb{L}_{k}(L, V)$.

The following important result is due to Murray-Stevenson-Vozzo \cite{Murray:2003,Murray:2010}, and was the motivation for much of this work. 
\begin{proposition}[\cite{Murray:2003,Murray:2010}]
Under the Caloron correspondence, let $(L,V) = \Cal(A,\kappa)$. Then the curvature of the lifting bundle gerbe $\mathbb{L}_{k}(L,V)$ computes the push-forward of the curvature of the Chern-Simons gerbe $\CSg_k(A)$, i.e.
\begin{equation}
\pi^{[\kappa]}_* \curv(\CSg_k(A)) = \curv(  \mathbb{L}_{k}(\Cal(\kappa,A)) \in H^3(B).
\end{equation}
\end{proposition}
Explicitly, Murray-Stevenson-Vozzo show that the curvature of the lifting bundle gerbe (when lifted to a basic form on $Q$) is given by  
\begin{equation}
\msv(L,V) :=  2\llangle F_\Lambda+d\lambda \Phi, \nabla_L \Phi  \rrangle \in \forms^3(Q,\Ft)_{\check \AG-basic} 
\end{equation}
and that
\begin{eqnarray}
\pi_Q^* \int_{Y/B} \langle F_A , F_A \rangle = \msv(L,V)
\end{eqnarray} 

With this we arrive at the main result of this section, which was somewhat implicit in the work of Murray-Vozzo, which provides us with the correct definition of the push-forward of the CS 2-gerbe.

\begin{proposition}
If we define $\check\pi^{\kappa}_* \CSg_k(A) := \mathbb{L}_{k}(\Cal(\kappa,A))$, then the following holds in $\check H^3(B)$
\begin{equation} \check\pi^{[\kappa]}_* \chr(\CSg_k(A)) = \chr( \check\pi^{\kappa}_* \CSg_k(A)) \end{equation}
I.e. the following diagram commutes
\begin{equation}
\begin{gathered}
\xymatrix{ 
\CSg_k(A)  \ar[r]^{\chr} \ar[d]^{\check\pi^{\kappa}_*}
 & \CS_k(A) \ar[d]^{ \check\pi^{[\kappa]}_*}\\ 
\mathbb{L}_{k}(\Cal(\kappa,A)) \ar[r]^{\chr} 
 & L_{k}(\Cal(\kappa,A)) \\ 
}
\end{gathered}
\end{equation}

\end{proposition}

\begin{proof}
Here we only prove a simplified version of this statement, in the case where $X$ is a boundary.
The statement above amounts to showing two things: Firstly, that the curvatures of the two gerbes agree, $\curv(\check\pi^{\kappa}_* \CSg_k(A) ) = \pi_*\curv( \CSg_k(A))$, which is the result of Murray-Vozzo. Secondly, we need to show that the holonomy functors agree, i.e. $\hol_{\CSg_k(A)}( \pi^{-1}(X) )  = \hol_{\mathbb{L}_{k}(\Cal(\kappa,A))} ( X )$. If $X$ is a boundary in $B$, i.e. $X = \partial Z$, then we have $\partial( \pi^{-1}(Z)) = \pi^{-1}(X)$, so
\begin{eqnarray}
\hol_{\mathbb{L}_{k}(\Cal(\kappa,A))} ( \partial Z ) &=& \exp 2\pi \Ri \int_{Z} \msv(\Cal(\kappa, A))\\
& =& \exp 2\pi \Ri \int_{\pi^{-1}(Z)} \langle F_A, F_A \rangle\\
& =& \hol_{\CSg_k(A)}( \pi^{-1}(X) ) 
\end{eqnarray}
The case where $X$ is not a boundary is can also be shown using more direct techniques, but is it somewhat cumbersome and will not be needed in this work.
\end{proof}

A direct consequence of the holonomy formula for the lifting bundle gerbe \ref{holliftgerbe} is
\begin{corollary}\label{holonomyoflifting}
The holonomy of the gerbe $\mathbb{L}_k$ is given by
\begin{equation}\label{BFhol}
\hol_{\mathbb{L}_{k}(L,V)}(\Sigma,\tilde Q,\alpha) = \exp 2\pi\Ri \smallint_\Sigma \lllangle F_{\tilde L}, \tilde V \rrrangle 
\end{equation}
\end{corollary}
Since this this holonomy is independent of the choice of lift $\tilde Q$, we see that the functional $\smallint_\Sigma \lllangle F_{\tilde L}, \tilde V \rrrangle $ is also independent of this choice modulo $\BZ$. We can also verify directly that as expected we have $d \lllangle F_{\tilde L}, \tilde V \rrrangle = \msv(L,V) \in \forms^3(Q)$, which is the analogue of the familiar relation $d\, cs(A) = p_1(A) \in \forms^4(P)$.

\section{The 2d Reformulation of Chern-Simons}\label{2dCSsection}

Upon looking at equation \ref{BFhol}, we immediately realize that the holonomy of the loop space lifting bundle gerbe is the Feynman amplitude for a BF type theory on the $\tilde \AG_k$-bundle $\tilde Q$. Recall as usual that $\tilde Q$ is a particular lift of the $\check \AG$-bundle $Q$. We call the field theory with this action the \emph{Caloron BF theory}. The action is
\begin{equation}
S_{k,\tilde Q}^{\Cal}(\tilde L, \tilde V) := \smallint_\Sigma \lllangle F_{\tilde L}, \tilde V \rrrangle,
\end{equation}
where the fields are a gauge field $\tilde L$, and an adjoint valued scalar field $\tilde V$ which is constrained to take in the adjoint orbit bundle $\orbit_{1,0} \in \tilde \Fg_k$. As we noted before, we usually write $\tilde V = \tilde V_\Phi = (1,-\Phi, \tfrac{k}{2} \llangle \Phi,\Phi\rrangle)$. This constraint for an adjoint scalar field is equally described by the equations 
\begin{equation}
\lllangle V,V \rrrangle = 0,\qquad \lllangle V, c \rrrangle = -1,
\end{equation}
where $c = (0,0,1)$ is the generator of the central extension. The fact that the adjoint scalar is constrained is a fundamental deviation from the usual BF theories, and is the cause for the unique behavior of Caloron BF theory.
We now arrive at the heart of our argument. When we implement fiber integration at the level of gerbes, we are pushing Chern-Simons forward \emph{as a TQFT}. We have shown that this push-forward is equivalent to Chern-Simons theory on $M$ using the technology of gerbes. However, in this section we will do some explicit checks of this equivalence. In particular, we will look at the classical equations of motion, and the classical symmetries of the 2d theory and show that they parallel those of the 3d theory.

\subsection{Classical Analysis}

Recall that the critical points 3d Chern-Simons action $\CS(A)$ on $P\to M$ are given by flat connections $A$ on $P$, i.e. $F_A = 0$. 
\begin{lemma}
The critical points of the Caloron BF action are given by solutions of the following equation in $ \forms^2(\Sigma, \Ad_{\!Q} \,\check \Fg)$
\begin{equation}
F_L = d\lambda\, V ,
\end{equation}
where $L = (\lambda,\Lambda)$ is the connection on $Q$.
\end{lemma}
\begin{proof}
We compute
\begin{eqnarray}
 \delta S_k^{\Cal}(\tilde L, \tilde V_\Phi) &=& \int_\Sigma\left( \lllangle \delta F_{\tilde L} , \tilde V_\Phi \rrrangle + \lllangle  F_{\tilde L} , \delta \tilde V_\Phi \rrrangle \right) \\
  &=& \int_\Sigma\left( \lllangle d_{\tilde L} \delta \tilde L , \tilde V_\Phi \rrrangle + \lllangle  F_{\tilde L} , (0,-\delta \Phi, k\llangle \delta \Phi, \Phi\rrangle) \rrrangle  \right) \\
  &=&- \int_\Sigma\left( \lllangle  \delta \tilde L , d_{\tilde L} \tilde V_\Phi \rrrangle + k\llangle \delta\Phi, F + d\lambda \Phi  \rrangle \right)  \label{impline}
\end{eqnarray}
where we write $F_L =: (d\lambda, F)$.
First of all, the variation in $\Phi$ demands $\beta( F_{ L},  V_\Phi):=F + d\lambda\Phi  = 0$. This can be expressed as 
\begin{equation}\label{maineq}
F_L = (d\lambda, - d\lambda \Phi) = V d\lambda,
\end{equation} so all we have to show is that the variation in $\tilde L$ imposes no new equations. Using \ref{maineq}, the Bianchi identity for $L$ gives $0=d_L V_\Phi  =: (0, \nabla \Phi)$.
Now, we have $\lllangle d_{\tilde L} \tilde V_\Phi ,\tilde V_\Phi\rrrangle = \tfrac{1}{2}d \lllangle \tilde V_\Phi ,\tilde V_\Phi\rrrangle =0 $, and so $d_{\tilde L} \tilde V_\Phi = (0,\nabla \Phi, k \llangle \Phi, \nabla \Phi\rrangle)$ and hence $d_L V = 0 \Leftrightarrow d_{\tilde L}\tilde V = 0$. Thus the variation in $ \tilde L$ imposes no new equations.
\end{proof}
Several observations are in order. First, notice that the critical locus $\mathrm{Crit}(S_k^{\Cal})$ only depends on the fields $L,V$ on $Q$ not on their lifts $\tilde L, \tilde V$ to $\tilde Q$. Second, we see that by virtue of the Bianchi identity for $F_L$, we see that the equation $d_L V=0$ is automatically satisfied for a classical solution, i.e, $V$ is covariantly constant. Thus, following the approach of \cite{Atiyah:1983}, we think of $V$ as a covariantly constant equivariant map $P \to \check \Fg$, and consider $Q_V = V^{-1}(d) \subset Q$ where $d=(1,0) \in \check \Fg$, which is a reduction of $Q$ to a bundle with structure group given by the centralizer $Z_d$. This centralizer $Z_d$ is given by $\BT \times G$. We see that $L$ reduces to a connection on $Q_V$ which we denote $\ell = (\kappa, A)$, with curvature given by the equation of motion $F_\ell = (d\kappa, 0)$. We see that we recover precisely the data of a flat connection $A$ on a trivial $G$-bundle, as expected from Chern-Simons theory. We will explore this relation more closely throughout this section.
Lastly we note that since $V$ was constrained to live in a particular adjoint orbit we found that the critical locus of this functional is described by a \emph{single} differential equation for the fields $L,V$, in stark contrast to the two equations we would find is this was an unconstrained BF theory for a connection and an adjoint scalar. One of the two equations becomes a consequence of the Bianchi identity.

Recall the definition of the looped connection in the Caloron correspondence: $\tilde A:= \beta( L , V) $, and we have $\tilde A_q = q^*A$ for a unique connection $A$ on $P$. Thus we have $(F_{\tilde A})_q = q^* F_A$. 
To compare the two critical loci, we need an expression for the curvature of this looped connection in terms of the data on $Q$.  \begin{lemma}
\begin{equation}\label{msvcurvature}
F_{\tilde A} = \beta(F_L, V) - \beta(L, d_L V) 
\end{equation}
\end{lemma}
\begin{proof}
We want to evaluate
\begin{eqnarray}
 F_{\tilde A} &=& d\tilde A + \tfrac{1}{2}[\tilde A,\tilde A] 
 \end{eqnarray}
Observe that the equivariance of $\beta$, gives $ [ X,\beta(Y,Z) ] = \beta([X,Y],Z)+ \beta( Y,[X,Z])$, which yields
\begin{eqnarray}
 [ \beta( X, Y) , \beta( W,Z) ] &=& -\,\beta( \beta([W,X],Y), Z) -\beta( \beta(X,[W,Y]), Z) \\
 && +\,\beta(\beta([Z,X],Y),W) + \beta(\beta(X,[Z,Y]),W).
 \end{eqnarray}
We also have $\beta([X,Y],V) = [X,Y]$, for any $X,Y \in \check \Fg$. So we compute
\begin{eqnarray}
[\tilde A, \tilde A] &=& [\beta(L , V) ,\beta(L , V) ] \\
 &=& \beta(\beta([L,L], V),V) - \beta(\beta(L,[L,V]),V) \\
 & & +\, \beta(\beta([V,L],V),L) + \beta(\beta(L,[V,V]),L)\\
&=& \beta([L,L], V)- 2\beta(L,[L,V]). 
\end{eqnarray}
Thus we have
\begin{eqnarray}
 F_{\tilde A} &=& d(\beta( L , V)) + \tfrac{1}{2}\beta([L,L], V)- \beta(L,[L,V])  \\
 &=& \beta( dL + \tfrac{1}{2}[L,L], V)- \beta( L , dV+[L,V])
 \end{eqnarray}
\end{proof}
We note a proof of this relation appears in some lengthy detail in \cite{Vozzo:2009}.
Now we can finally express the full classical equivalence of Chern-Simons on $M$, and Caloron BF on $\Sigma$.

\begin{proposition}
The Caloron correspondence, which is a bijection between connections on $P\to M\to \Sigma$ and connections and Higgs fields on $Q \to \Sigma$, induces a bijection between the spaces of classical solutions for the functionals $S^\CS_k$ and $S^{\Cal}_k$.
\end{proposition}
\begin{proof}
Recall, that under the Caloron correspondence, the lift of the $G$-connection is defined by  $q^* A  = \tilde A := \beta(L,V) = L - \lambda V$. We can easily see that $q^* F_A = F_{\tilde A} = F_{L-\lambda V} = F_{L} - d\lambda V + \lambda d_L V  = 0 $.
Now, pick a pair of critical data $(L,V)$ on $Q$, i.e. satisfying $\beta( F_{ L},  V) = 0$, and hence also $d_L V = 0$. Then consider the corresponding looped connection $\tilde A$, defined as usual by $\tilde A = \beta( L, V)$. By comparing with formula \ref{msvcurvature}, $F_{\tilde A} = \beta(F_L, V) - \beta(L, d_L V) $, we find that $F_{\tilde A} = 0$, which implies $F_A = 0$. On the other hand starting with a flat connection $A$ on $P$, the corresponding looped connection satisfies $F_{\tilde A} = 0$, which we see forces $\beta(F_L, V) - \beta(L, d_L V)=0$. Now of the two components in this expression, only the first is horizontal, and $\beta(L, d_L V) = \lambda \nabla \Phi$ is purely vertical. Thus each of $\beta(F_L, V)$ and $\beta(L, d_L V)$ must independently be zero.\end{proof}

Note that the Chern-Simons functional does not involve the choice of connection $\kappa$, however the Caloron BF functional does.

\subsection{Symmetries of the Classical Action}

Recall that the Caloron BF action $S^{\Cal}_k(\tilde L, \tilde V) = \int_\Sigma \lllangle F_{\tilde L},\tilde V\rrrangle$ was defined in terms of a particular chosen lift $\tilde Q$ of the equivariant loop space $Q$ of the $G$-bundle $P \to M$. This action is manifestly invariant under the gauge group $\gauge_{\tilde Q}$, since both of $F_{\tilde L}$ and $\tilde V$ are adjoint valued fields. We will refer to this as the group of `small' gauge transformations. It is analogous to the situation in Chern-Simons theory, where the group of gauge transformations that are in the connected component of the identity are `small', and the classical action is invariant under them. However, we will show that the Caloron theory also has a larger symmetry group given by the gauge transformations of the $\check \AG$-bundle $Q$, i.e. $\gauge_{Q}$, including those which \emph{do not} have a lift to an automorphism of $\tilde Q$. We will see that these `large' gauge transformations act to change the choice of lift $\tilde Q$, and will demonstrate that the underlying theory was truly independent of this choice.

Note that for any particular lift $\tilde Q$, there is a non-surjective group homomorphism $i_{\tilde Q} : \gauge_{\tilde Q} \to \gauge_{ Q}$, given by forgetting the action on the lift. The kernel of this homomorphism is the gauge transformations of $\tilde Q$ that fix $Q$, which are given by smooth maps from the space of $\check \AG$ orbits in $Q$ to $\BT_c$, i.e. $\Map(\Sigma,\BT_c)$. Consider the central extension $\tilde \AG_k \to \check \AG$ as a $\check \AG$-equivariant line bundle over $\check \AG$, where $\check \AG$ acts on itself via conjugation. The equivariant Chern class $c_1^{\check \AG}(\tilde \AG)$ lies in the equivariant cohomology $H^2_{\check \AG}(\check \AG,\BZ)$. Thus for a large gauge transformation $\psi$, represented by an equivariant map $g_\psi: Q \to \check \AG$, we have $g^*_\psi c^{\check \AG}_1(\tilde \AG) \in H^2_{\check \AG}(Q,\BZ) \cong H^2(\Sigma,\BZ)$.
This set of maps fit into an exact sequence
\begin{equation}
1\to \Map(\Sigma,\BT_c) \to \gauge_{\tilde Q} \to \gauge_{Q} \to H^2(\Sigma, \BZ)
\end{equation}
We find that the group $  \gauge_{ Q}$ of large gauge transformations acts on the set of lifts of $Q$, which is a torsor for the Picard group of $\Sigma$.

\begin{lemma}\label{largegaugeshift}
Under a gauge transformation $\psi \in \gauge_Q$, we have
\begin{equation}\label{lifttfn}
\psi^*:  \tilde Q \mapsto \tilde Q \otimes g^*_{\psi} \tilde \AG 
\end{equation}
as line bundles over $Q$. Furthermore, a lift $\tilde \psi$ of $\psi$ gives a section of $g^*_{\psi} \tilde \AG$.
\end{lemma}
\begin{proof}
Since, $\tilde Q$ is a $\tilde \AG$ bundle, it is clear that $\Hom_\BT(\tilde Q_q, \tilde Q_{\psi(q)}) \cong \tilde \AG_{g_{\psi}(q)}$. This gives $\tilde Q^\vee \otimes \psi^* \tilde Q \cong g_\psi^* \tilde \AG$.
A lift $\tilde \psi$ gives isomorphisms $\tilde \psi_{q} : \tilde Q_{q} \to \tilde Q_{\psi(q)}$, i.e. a section of $\tilde Q^\vee \otimes \psi^*\tilde Q$.
\end{proof}

Now we explicitly demonstrate that $S^{\Cal}_k$ modulo $\BZ$ is independent of the choice of lift, thus invariant under $\gauge_Q$. In fact, the quantum symmetry group is larger: we can tensor our lift by any line bundle with connection on $\Sigma$, and this will leave our action invariant.  The rest of this section will be devoted to proving the following.
\begin{proposition}\label{invarianceprop}
The function 
\begin{equation} S^{\Cal}_{k} (L,V) := S^{\Cal}_{k,\tilde Q}(\tilde L, \tilde V) \qquad \mod \BZ \end{equation} is independent of the particular choice of lifting data $(\tilde Q, \alpha)$, and depends only on the $\gauge_Q$-equivalence class of $(L,V)$ and the level $k$ of the lift.
\end{proposition}

Choose a lift $\tilde Q, \tilde L$, and let $(\ell, \beta)$ be a line bundle with connection on $\Sigma$. We pull up $\ell$ to $Q$, where it becomes a $\check \AG$ invariant line bundle with connection. The tensor product $\tilde Q \otimes \ell$ becomes a new lift of $Q$, and $\alpha \otimes 1 + 1 \otimes \beta$ is a connection on this lift, which gives a different lift $\tilde L '$ of the connection $L$. Thus the space of lifts $(\tilde Q, \alpha)$ is a torsor for $\mathrm{Pic}(\Sigma)$. The curvature of this new connection is $F_{\tilde L'} = F_{\tilde L} + d\beta$. So we see 
\begin{equation} \Psi_{\tilde Q \otimes\ell}( \tilde L', \tilde V) = \int_\Sigma \lllangle F_{\tilde L} + d\beta, \tilde V\rrrangle = \Psi_{\tilde Q }( \tilde L, \tilde V) + \int_\Sigma d\beta\end{equation}
since $s_{V}(\cdot) = \lllangle\cdot, \tilde V\rrrangle$ is a bundle splitting splitting. Since $d\beta$ was the curvature of a line bundle, we have $\int_\Sigma d\beta \in \BZ$.

Now we note that the action of a (large) gauge transformation of $Q$ on a particular lift $\tilde Q$, can be decomposed as a twist of the lift by a line bundle on $\Sigma$, composed with a (small) gauge transformation of the twisted lift. If we perform a large gauge transform, $\psi$, then we have $\psi^*\tilde Q \cong \tilde Q \otimes g_{\psi}^* \tilde \AG$, by lemma \ref{largegaugeshift}. Then tensor by a line bundle $\ell$ on $\Sigma$, with first Chern that cancels the class given by $g_{\psi}^* \tilde \AG$. Then there exists an isomorphism of equivariant line bundles $\ell \otimes \psi^*\tilde Q \cong \tilde Q$. Thus we have proven proposition \ref{invarianceprop}.

Note that this is precisely analogous to the situation in 3d Chern-Simons theory on bundles $P$ that are trivializable. If we choose a section $s$ of $P$, then we can define the Chern-Simons action using this section. If we pick a different section $s'$, then the two sections are related by a large gauge transformation. In this way, we see that the set of large gauge transformations acts transitively on the set of sections modulo small gauge transformations. Thus the fact that Chern-Simons action for different trivializations differs by an integer is reflected in the fact that the Caloron BF action for different lifts of $Q$ differs by an integer. Reinforcing the fact that lifts of $Q$ \emph{are} trivializations in the bundle gerbe sense (see appendix).

To summarize, in order to write a classical action for the theory, we need to choose a lift, but all such choices are related by symmetry. Thus $\gauge_Q$ is the group of symmetries of the theory, where we consider all possible lifts simultaneously. Only those gauge transformations the possess lifts are redundancies of the system, the rest become symmetries of the quantum theory. 

\subsection{The Hamiltonian Structure}\label{SymplecticStructure}
Beasley-Witten show that the (reduced) space of connections in Chern-Simons theory on $M \to \Sigma$ has a symplectic structure. Here we introduce the analogue of this symplectic form on the Caloron side. The space of connections $\connections_{\tilde Q}$ on $\tilde Q$ is naturally symplectic, as described by the Atiyah-Bott \cite{Atiyah:1983} construction. The symplectic form is 
\begin{equation} \Omega = \tfrac{1}{2}\int_{\Sigma}\lllangle \delta \tilde L \^ \delta \tilde L \rrrangle \end{equation}
The gauge group $\gauge_{\tilde Q}$ of $\tilde Q$, has a Hamiltonian action on this space with moment map $\mu_\connections(\tilde L) = F_{\tilde L} \in \liegauge_{\tilde Q}^\vee$. As usual, the Lie algebra $\liegauge_{\tilde Q}$ can be equipped with a non-degenerate invariant inner product if we choose a volume form $\nu$ on $\Sigma$. However the moment map $\mu_\connections: \connections_{\tilde Q} \to \liegauge_{\tilde Q}^\vee$ is independent of this choice.

Secondly, we consider the space $\Higgs_{\tilde Q}$ of Higgs fields. This space is given by the space of sections of an adjoint orbit bundle and (co)adjoint orbits $\orbit$ are naturally symplectic with the Kostant-Souriau symplectic form $\Omega_{\orbit}$. However this symplectic form is only fiber-wise symplectic on $\Higgs_{\tilde Q}$. However, using the same choice of volume form $\nu$ on $\Sigma$ we can define a symplectic form
\begin{equation}
\Omega_\Higgs = \int_\Sigma  \Omega_{\orbit}\, \nu
\end{equation}
The (dual) moment map for this action $\mu_\Higgs^\vee : \Higgs_{\tilde Q} \to \liegauge_{\tilde Q}$ is given by inclusion into the Lie algebra, and is independent of the choice of volume form.

We have constructed two Hamiltonian $\gauge_{\tilde Q}$-spaces, the space of connections $\connections_{\tilde Q}$, and the space of Higgs fields $\Higgs_{\tilde Q}$. We now notice that the Caloron BF action can be expressed as the canonical pairing of these two moment maps
\begin{equation} \int_{\Sigma} \lllangle F_{\tilde L}, \tilde V \rrrangle  = (\mu_{\connections}, \mu_{\Higgs}^\vee) \end{equation}

Here we will give a brief discussion about how the path integral of this theory can be evaluated using the techniques of non-abelian localization (i.e. norm-squared localization). Fundamentally, the calculations will be the analogous as those done in Beasley-Witten \cite{Beasley:2005}. However, one major subtlety that we ignore for now is that the path integral should only be integrating over all connections on $Q$, \emph{not} over the space of their lifts to $\tilde Q$.

In the foundational paper of Witten \cite{Witten:1992a}, a certain class of integrals over symplectic manifolds are shown to be calculable using localization techniques. Given a Hamiltonian manifold $(X,\omega, \mu)$, one is interested in evaluating the following integral
\begin{equation} Z_X(\varepsilon) = \int_{X} e^{- \tfrac{1}{2\varepsilon} |\mu|^2} e^\omega \end{equation}
where $\mathcal{L} = e^{\omega}$ is the Liouville volume form, and the norm is given by the non-degenerate invariant pairing on $\Fg^*$.
By un-completing the square, this integral is written in the following way
\begin{equation} Z_X(\varepsilon) = c_\varepsilon \int_{X\times \Fg^*} [d\xi] e^{-\frac{\varepsilon}{2} |\xi|^2} e^{ \Ri \langle \xi,\mu\rangle} e^\omega = c_\varepsilon \int_{ \Fg^*} [d\xi] e^{-\frac{\varepsilon}{2} |\xi|^2} \int_X e^{\Ri \langle \xi,\mu \rangle} e^\omega \end{equation}
where $[d\xi]$ is the invariant measure on $\Fg^*$. In this expression, we see the appearance of the \emph{Duistermaat-Heckman} measure $\mathrm{DH}_X  \in \mathcal{D}'(\Fg^*)^G$, given by
\begin{equation} \mathrm{DH}_X(\xi) := \int_X e^{\Ri \langle \xi,\mu\rangle} e^\omega \end{equation}
which is a $G$-invariant distribution on $\Fg^*$. If we let $f_\varepsilon(\xi) =  c_\varepsilon e^{-\frac{\varepsilon}{2}|\xi|^2}$ be the $G$-invariant gaussian function on $\Fg^*$ with unit integral, then we see that
\begin{equation} Z_X(\varepsilon) = ( f_{\varepsilon}, \mathrm{DH}_X)_{\Fg^*} \end{equation}
where the brackets now denotes the pairing between distributions and test functions on $\Fg^*$. Once the integral has been recast in this form, localization techniques can then be used to get an explicit form of the DH measure. Often in cases where $X$ is infinite dimensional, $\mathrm{DH}_X$ cannot be defined explicitly, but instead a suitable replacement of its localization can be used to give meaning to $Z_X$ (c.f. \cite{Woodward:2005}). This is precisely what was done in \cite{Witten:1992a}, where the path integral for 2d Yang-Mills
was computed by deforming the calculation of the measure $\mathrm{DH}_{\connections_P}$ using a technique known as norm-squared localization.

However, the integral we are interested in evaluating is of a slightly different form, nevertheless the same techniques allow for it to be localized. Suppose we are given two Hamiltonian $G$-manifolds $(X, \omega_X, \mu_X)$, $(Y,\omega_Y, \mu_Y)$, and we want to evaluate the following integral
\begin{equation} Z_{X,Y} = \int_{X \times Y}e^{\Ri(\mu_X,\mu_Y^*)} e^{\omega_X} e^{\omega_Y} \end{equation}
where we have written $\mu_Y^*$ as the dual of the moment map of $Y$, so that the pairing $(\mu_X,\mu_Y^*)$ is the canonical one.
We use the identity
\begin{eqnarray}
e^{\Ri(\mu_X,\mu_Y^*)} &=&\int_{\Fg^*\times \Fg} [d\eta d\xi] e^{-\Ri (\eta,\xi)} e^{ \Ri \langle\eta,\mu_X\rangle} e^{\Ri\langle \xi,\mu_Y^*\rangle}
\end{eqnarray}
to find
\begin{equation} Z_{X,Y} = \int_{\Fg^*\times \Fg} [d\eta d\xi] e^{-\Ri (\eta,\xi)} \mathrm{DH}_X(\eta) \mathrm{DH}^*_Y(\xi) \end{equation}
where
\begin{equation} \mathrm{DH}^*_{Y} \in \mathcal{D}'(\Fg)^{G} \end{equation}
is the dual DH measure.
Thus we see that a more explicit form of $Z_{X,Y}$ can be determined if one knows both the DH measures of $X$ and $Y$. Importantly, if it is the case that either of these DH measures can be evaluated explicitly using localization techniques, the whole integral should reduce. In the case where $X,Y$ are infinite dimensional, we use the same techniques to define the measures and hence give meaning to $Z$. In this way, we can view the calculations in Beasley-Witten as a computation of the localization formula for the measure $\mathrm{DH}_{\connections_{\tilde Q}}$, which allows for explicit evaluation of this path integral.

\section{The Beasley-Witten Construction Revisited}\label{BWRevisited}
In this section, we will use the constructions of the Caloron correspondence to shed new light of the results of Beasley-Witten outlined in section \ref{BWsection}. Our initial goal will be to construct a dictionary relating the key components of the Caloron data with elements of the Bealsey-Witten construction. We again assume $P$ is trivialized by a section $s$, so that $P \cong M\times G$, however the bundle $M \to \Sigma$ may be non-trivial. We will find that in this situation, the constructions of the Caloron correspondence have `unlooped' avatars in the 3 dimensional world, which do not require the introduction of the equivariant loop space. Given a connection form $\kappa \in \forms^1(M,\Ft)$ for the circle bundle $\pi : M \to \Sigma$, the projection operator $ \bar J_{\kappa} :\alpha \to \alpha - \kappa \wedge (\iota_R \alpha)$ decomposes the space of connections $\connections_P \cong \forms^1(M, \Fg)$ in the following way:
\begin{equation}\label{BWdecomp}
\connections_P \cong \bar J_\kappa \connections_P \oplus \iota_R \connections_P, \quad A \mapsto (a,\phi) := (A-\kappa\iota_R A, \iota_R A)
 \end{equation}
As before, we denote $\bar\connections_\kappa := \bar J_\kappa \connections_P$. Note that we do not have such a decomposition when the bundle $P$ is non-trivial.

Beasley-Witten consider the action of the semi direct product of two symmetry groups on the space $\bar \connections_\kappa$. Namely, the action of rigid $\BT$ transformations of $M$ combined with gauge transformations of $P$, i.e. $\liegauge'= \Ft \oplus \Fg^M$. Here we will now enlarge this by considering \emph{local} $\BT$ transformations, i.e. gauge transformations of $M \to \Sigma$. Under such a local transformation, the connection $\kappa$ is no longer invariant, as it transforms as a gauge field. So we are forced to consider the collection $\bar \connections = \sqcup_{\kappa \in \connections_M} \bar \connections_\kappa$. This space is a bundle over $\connections_M$, topologized as a sub-bundle of the trivial bundle $\connections_M \times \connections_P$. Consider now the larger symmetry algebra  $\liegauge = \liegauge_M \oplus \liegauge_P \cong \Ft^\Sigma \oplus \Fg^M$ with bracket
\begin{equation} [(x_1, \xi_1),(x_2,\xi_2)] = (0, [\xi_1,\xi_2] -  \LieDer_{x_1 R} \xi_2 + \LieDer_{x_2 R} \xi_1 ).\end{equation}
We see that $\liegauge' \subset \liegauge$ via the inclusion of constant maps $\Ft \subset \Ft^\Sigma$.
As before (c.f. \ref{barAaction}), we can check that the space of pairs $(\kappa, a) \in \bar\connections$ carries a representation of this larger algebra 
\begin{equation}\label{BASICREP} \iota_{(x, \xi)} \delta a = D_a \xi + \LieDer_{x R} a, \qquad \iota_{(x,\xi)} \delta \kappa = dx \end{equation}
 If we define the `connection' form $\ell = (\kappa, a) \in \forms^1(M, \Ft \oplus \Fg)$, then the above transformations can be expressed as
\begin{eqnarray}
\iota_{(x,\xi)} \delta\ell &=& d(x,\xi) + [\ell,(x,\xi)] \\
&=& d_\ell(x,\xi) 
 \end{eqnarray}
Furthermore $\iota_{ R}\ell = (R,0)$, where $R$ is the generator of the global $\Ft$ action on $M$.
Thus we call this the \emph{gauge} representation of $\liegauge$, and we call $\ell$ a $c$-connection. This will be the BW analogue of our Caloron connection $ L = (\lambda,\Lambda) \in \connections_Q$.

The other half of the decomposition \ref{BWdecomp} is given as the space of functions $\phi = \iota_R A \in \Fg^{M}$ for $A \in \connections_P$. Under a gauge transformation of $P$ we have $\phi \mapsto \Ad_{g^{-1}} \phi + g^{-1} \LieDer_R g$. Infinitesimally this is $\iota_\xi \delta \phi =- [\xi,\phi] + \LieDer_R \xi$. On the other hand, under a circle diffeomorphism of $M$, we find $\iota_{x}\delta \phi = \LieDer_{xR} \phi$.
We can simply express the two transformations above by introducing the element $v_\phi = (1,-\phi) \in \liegauge $. With this definition, we can easily see that the representation of the algebra $\liegauge$ on this field is the adjoint 
\begin{equation}\label{hfieldtransform}
\iota_{(x,\xi)} \delta v_\phi =\ad_{(x,\xi)} v_\phi.
\end{equation}
We also call such an element a $h$-field, as it the BW version of our Higgs fields.

Thus we have established a few entries of our desired dictionary 
\begin{equation}\label{bwdictionary}
\begin{gathered}
\begin{array}{c c|c c }
\text{Beasley-Witten data}\qquad&&&\qquad\text{Caloron data}\\
\hline
\text{gauge field}& A \in \connections_P & (0,\tilde A) & \text{Looped connection} \\
\text{contact form}& \kappa \in \connections_M & ( \tilde \kappa,0)& " \qquad" \\
\text{$c$-connection} &(\kappa,a) \in \bar \connections &   (\lambda,\Lambda) \in \connections_Q & \text{gauge field on $Q$}   \\
\text{$h$-field}& (1,-\phi) &   (1,-\Phi) & \text{Higgs field}\\
\end{array}
\end{gathered}
\end{equation}
At this point, we attempt to mimic the construction of the Symplectic form on $\bar\connections_\kappa$ that was used by Beasley-Witten. To this end, we ask if we can find a symplectic form on the total space $\bar\connections$ for which the action \ref{BASICREP} (or perhaps a central extension of it) is Hamiltonian. To answer this, we again seek motivation from the Caloron correspondence. 

\subsection{The Central Extension}

In section \ref{SymplecticStructure}, it was shown that there is a natural symplectic form on the space of connections on the lifted bundle $\tilde Q$. To mimic this construction in the BW framework, we look for a covering $\tilde \connections \to \bar \connections$ that is the shadow of the cover $\connections_{\tilde Q} \to \connections_Q$. First, we construct the analogue of the central extension of the gauge group $\gauge_{\tilde Q} \to \gauge_Q$, i.e. a central extension of $\tilde \liegauge$. Let $K : C^\infty(M) \to C^\infty(M)^{\BT} \cong C^{\infty}(\Sigma)$ be the $\BT$-averaging operator
\begin{equation} (K v)(x) = \tfrac{1}{2\pi} \int (r_\theta^* v)(x)d\theta \end{equation}
Futhermore, introduce the $\Ft^\Sigma$-valued pairing on $\Fg^M$, $\llangle \xi,\xi' \rrangle := K\langle \xi, \xi' \rangle$.

\begin{lemma}
The map $\gamma : \Fg^M\times \Fg^M \to \Ft^\Sigma$ given by $\gamma(\xi_1,\xi_2) = \llangle  \xi_1, \LieDer_R\xi_2 \rrangle$ is a 2-cocycle on $\liegauge$ with values in $ \Ft^\Sigma$, and thus defines a central extension
\begin{equation} 0 \to  \Ft^\Sigma \to \tilde \liegauge \to \liegauge \to 0 \end{equation}
\end{lemma}
\begin{proof}
It is easily checked that $\gamma([\xi_1,\xi_2],\xi_3) + cyclic = 0$. Also, $\gamma([x_1,\xi_2],\xi_3) = \gamma(\LieDer_{x_1R} \xi_2,\xi_3)$
$=-\gamma( \xi_2,\LieDer_{x_1R} \xi_3) = -\gamma(\xi_2,[x_1,\xi_3])  $, since $[\LieDer_{x_1R},\LieDer_{R}] = 0$. More simply, this follows from the fact that $\LieDer_R$ is a derivation.
\end{proof}
In a completely analogous proof as in the case of loop groups \cite{Pressley:1986}, we find that the following $\Ft^\Sigma$-valued pairing on $\tilde \liegauge$
\begin{equation} \lllangle (x_1,\xi_1,y_1) , (x_2,\xi_2,y_2) \rrrangle  = \llangle \xi_1,\xi_2\rrangle - x_1 y_2 - x_2 y_1  \end{equation}
is invariant and non-degenerate.

Now we have constructed our central extension, we move on to finding a gauge theoretic representation that extends the representation (\ref{BASICREP}). Recall that in the Caloron correspondence, we lifted the connection $L = (\lambda, \Lambda)$ to a connection $\tilde L = (\lambda, \Lambda, \alpha)$ on $\tilde Q$. So we accordingly extend $\ell \in \bar\connections$ to a form ${\tilde\ell} = (\kappa, a,b) \in \forms^1(M, \Ft_R \oplus \Fg \oplus \Ft_c)$, where $b$ is a $\BT_R$-basic form, with the following infinitesimal action of $\tilde \liegauge$
\begin{eqnarray}\label{GAUGEREP}
 \iota_{(x,\xi,y)} \delta {\tilde\ell} &=& d_{\tilde\ell}(x,\xi,y)  \\
&=& d(x,\xi,y) + [{\tilde\ell},(x,\xi,y)] 
\end{eqnarray}
We call such a form $\tilde \ell$ a $\tilde c$-connection, and denote the space of all such objects $\tilde \connections$. Looking at the representation above, we see that the action on the $b$ component is
\begin{equation}\label{BFIELDREP}
 \iota_{(x,\xi,y)} \delta b = dy + \llangle  a,\LieDer_R \xi \rrangle \in \forms^1(\Sigma, \Ft) 
 \end{equation}
We see that $b$ looks like a $\BT$ gauge field on $\Sigma$ for the vector fields $\Ft_c^\Sigma$, however it transforms non-trivially under a $G$-gauge transformation of $P$. 

We can easily see that $\bfields$ is a trivial $\forms^1(\Sigma, \Ft)$ bundle over $\bar\connections$, and from the definition \ref{GAUGEREP} we see that the above action of $\tilde \liegauge$ on $\bfields$ is a representation.
In the Caloron construction, we found that the space of connections $\tilde L$ on $\tilde Q \to \Sigma$ came equipped with the canonical Atiyah-Bott (c.f.  \cite{Atiyah:1983}) symplectic form 
\begin{equation}
\Omega = \tfrac{1}{2}\int_\Sigma \lllangle \delta\tilde L, \delta \tilde L \rrrangle
\end{equation}
This suggests to us that the space $\bfields$ of $\tilde c$-connections should be equipped with a symplectic form
\begin{equation} \Omega_{\bfields} = \tfrac{1}{2} \int_\Sigma \lllangle \delta {\tilde\ell} \^ \delta {\tilde\ell} \rrrangle= \int_{\Sigma}\left(\tfrac{1}{2} \llangle \delta a \^ \delta a\rrangle - \delta \kappa \^ \delta b\right) . \end{equation}
This leads to
\begin{proposition}
The form $ \Omega_{\bfields}$ is a $\tilde \liegauge$-invariant symplectic form on $\bfields$. Furthermore, the action is Hamiltonian with moment map
\begin{equation}
\mu_{{\tilde\ell}}= F_{\tilde \ell} := d{\tilde\ell} + \tfrac{1}{2}[{\tilde\ell},{\tilde\ell}]= (d\kappa, F_a - \kappa \LieDer_R a,db + \tfrac{1}{2} \llangle a, \LieDer_R a\rrangle) 
\end{equation}
We see that $F_{\tilde \ell} \in \forms^2(M, \Ft \oplus \Fg \oplus \Ft)_{horiz}$.
\end{proposition}
\begin{proof}
This follows directly from the formula \ref{GAUGEREP}, and the $\ad$-invariance of the pairing.
\end{proof}

So we can see that $\mu$ can be thought of as the curvature of the $\tilde c$-connection ${\tilde\ell}$, in exactly the same way that the moment map in the Caloron picture is the curvature. We can already begin to see the similarity between this moment map and the moment map of Beasley-Witten (eq \ref{BWMOMENT}). In the way that we extended the space of $c$-connections, we also define $\tilde h$-fields as the elements of the $G^M$-adjoint orbit of $(1,0,0) \in \tilde \liegauge$. These are of the form
\begin{equation} \tilde v_\phi =(1,-\phi,\tfrac{1}{2}\llangle \phi,\phi\rrangle) \in \tilde\liegauge\end{equation}
and are clearly in bijective correspondence with regular $h$-fields. This allows us to extend the dictionary
\begin{equation}\label{biggerbwdictionary}
\begin{gathered}
\begin{array}{c c|c c }
\text{Beasley-Witten data}\qquad&&&\qquad\text{Caloron data}\\
\hline
\text{ $\tilde c$-connection} &\tilde \ell = (\kappa,a,b) \in \tilde \connections &  \tilde L = (\lambda,\Lambda,\alpha) \in \connections_{\tilde Q} & \text{gauge field on $\tilde Q$}   \\
\text{$\tilde h$-field}& (1,-\phi,\tfrac{1}{2}\llangle \phi,\phi\rrangle) &   (1,-\Phi,\tfrac{1}{2}\llangle \Phi,\Phi\rrangle) & \text{Higgs field}\\
\text{symplectic form}& \Omega_{\bar \connections} &   \Omega_{\connections_{\tilde Q}} & \text{symplectic form}\\
\text{moment map}& F_{\tilde \ell} &   F_{\tilde L} & \text{moment map}\\
\end{array}
\end{gathered}
\end{equation}
We found that in the Caloron picture, the action functional was given by the 2d BF form $S = \lllangle F_{\tilde L} , \tilde v\rrrangle$. Here we explicitly compute
\begin{eqnarray}
\lllangle F_{\tilde \ell}, \tilde v_\phi\rrrangle = - \llangle F_a- \kappa \LieDer_R a ,\phi \rrangle - \tfrac{1}{2} d\kappa \llangle \phi,\phi\rrangle - \tfrac{1}{2} \llangle a, \LieDer_R a\rrangle- db \in \forms^2(M, \Ft)_{basic}
\end{eqnarray}
It is easy to check that
\begin{equation}
\lllangle F_{\tilde \ell}, \tilde V_\phi\rrrangle = -\iota_R K\left(\langle \phi, \kappa \^F_a - d\kappa \^a \rangle + \tfrac{1}{2} \kappa\^ d\kappa \langle \phi,\phi\rangle -\tfrac{1}{2}  \langle a, d a\rangle + \kappa \^ db \right) 
\end{equation}
To integrate this form over $\Sigma$, we use the following identity
\begin{equation}\label{fiberint}
\int_{M} \alpha = \int_{\Sigma} \iota_R K\alpha 
\end{equation}
i.e., to integrate over the circle fiber, we map to the $\BT$ basic forms by averaging and contracting. Thus
\begin{equation}
\int_\Sigma \lllangle F_{\tilde \ell}, \tilde v_\phi\rrrangle = -\int_M \left(\langle \phi, \kappa \^F_a - d\kappa \^a \rangle + \tfrac{1}{2} \kappa\^ d\kappa \langle \phi,\phi\rangle +\tfrac{1}{2} \kappa \langle a, \LieDer_R a\rangle + \kappa \^ db \right)
\end{equation}
We can see here, that upon performing the gaussian integration over $\phi$, we recover exactly the square of the Beasley-Witten moment map \ref{BWMOMENT} (up to terms involving $db$).
\subsection{Inclusion of Wilson Loops}
We discussed earlier that Beasley has also shown that when Wilson loop operators that wrap a single $\BT$ fiber are included the path integral still possesses a localization formula. However, the moment map that was found by Beasley is somewhat complicated. Recall that the Wilson loop around a single fiber $C = \pi^{-1}(z)$ can be expressed as 
\begin{equation}
W_\alpha(C,A) = \int_{Map(C, \orbit_\alpha)} \mathcal{D}U \exp\left(\Ri \,\cs_\alpha(U,A) \right) 
\end{equation}

We will present a slightly different formulation of this result that appears much more naturally in the setting of the Caloron framework. For a $\Fg$-coweight $\alpha^\vee  \in \Fh$, consider the $\tilde \Fg _k$-coweight ${\tilde \alpha^\vee } := (0,\alpha^\vee ,0) \in \tilde \Fh$. For $g \in LG$, we define
\begin{equation} \mathcal{U} := \Ad_g \tilde \alpha = (0, g \alpha^\vee g^{-1}, k\llangle \alpha^\vee, g^{-1} \partial g \rrangle ) \in \tilde \Fg \end{equation}
where we have used equation (\ref{affinegadjoint}). The set of all such $\mathcal{U}$ is an $LG$ adjoint orbit $\Ad_{LG}(\tilde \alpha) \subset \tilde \Fg$. Notice that this space is isomorphic to the full $\check \AG$ adjoint orbit $\orbit_{\tilde \alpha} \subset \tilde \Fg$, since the rotation subgroup acts in a semi-direct fashion on $LG$.

Recall our $h$-fields, which are the vertical components of the connections, are given by $v_\phi = (1,-\phi,\tfrac{1}{2}\llangle \phi, \phi \rrangle)$, where $\phi = \iota_R A$. We then observe that
\begin{eqnarray}
 k^{-1}\lllangle \mathcal{U}, v_\phi \rrrangle &= &-\llangle \phi, g \alpha^\vee g^{-1}\rrangle -  \llangle \alpha^\vee, g^{-1} \partial g \rrangle \\ 
&=& - \cs_\alpha(U,A) 
\end{eqnarray}
This shows that the functional $\cs_\alpha(U,A) $ is can be expressed as a manifestly $\tilde \AG_k$ invariant quantity. So we then have the formal equality of path integrals
\begin{equation}\label{holformula} \int_{Map(C, \orbit_\alpha)} \mathcal{D}U \exp\left(\Ri \,\cs_\alpha(U,A) \right) = \int_{\orbit_{\tilde \alpha}} \mathcal{D}\mathcal{U} \exp\left( -\Ri k^{-1} \,\lllangle \mathcal{U}, v_\phi \rrrangle \right) \end{equation}

We can extend this construction very easily to the situation of the Caloron BF functional.\begin{definition}
For $z \in \Sigma$, and an integral weight $\alpha \in \Fh$, define the Wilson loop observable
\begin{equation} W_{\alpha}(z, \tilde V) = \int_{\orbit_{\tilde \alpha}} \mathcal{D}\mathcal{U} \exp\left( -\Ri k^{-1} \,\lllangle \mathcal{U}, \tilde V(z) \rrrangle \right) \end{equation}
where $\orbit_{\tilde \alpha} \subset \tilde \Fg_k$ is the adjoint orbit of $\tilde \alpha =(0,\alpha,0) \in \tilde \Fh$. 

\end{definition}

Beasley \cite{Beasley:2013} showed that the expectation values of these Wilson loop operators can also be calculated using equivariant localization. We will now show that this result becomes manifest in our loop-space picture.
We can consider the observable $W_\alpha(z)$ as a function on $\liegauge_{\tilde Q}$. In terms of the symplectic picture of section \ref{SymplecticStructure}, we have

\begin{eqnarray}
Z(k,\alpha,z) &:=& \int_{\connections_{\tilde Q} \times \Higgs_{\tilde Q} } \mathcal{D}(\tilde L, \tilde V)  W_\alpha(z,\tilde V) \exp\left({2\pi \Ri \int_{\Sigma} \lllangle F_{\tilde L}, \tilde V\rrrangle }\right) \\
&=&  \int_{\connections_{\tilde Q} \times \Higgs_{\tilde Q} \times \orbit_{\tilde \alpha} }  \mathcal{D}(\tilde L, \tilde V,\mathcal{U})  \exp\left({2\pi \Ri \int_{\Sigma} \lllangle F_{\tilde L}+ k^{-1} \delta(z) \mathcal{U}, \tilde V\rrrangle }\right)
\end{eqnarray}
Since the adjoint orbit $\orbit_{\tilde \alpha}(z)$ is automatically a Hamiltonian $\gauge_{\tilde Q}$-space, with moment map $\delta(z) \mathcal{U}$, we expect that this path integral should factorize into components
\begin{equation} \int_{\liegauge^* \times \liegauge} [d\xi\, d\eta] e^{-\Ri( \xi,\eta)} \mathrm{DH}_{\connections_{\tilde Q} \times \orbit_{\tilde \alpha}}(\xi) \mathrm{DH}^*_{\Higgs_{\tilde Q}}(\eta) \end{equation}
This factorization can be seen in the results of BW, where it is shown that in the presence of a Wilson loop at $z$, the path integral measure localizes around singular Yang-Mills connections with monodromy $\exp(\Ri \alpha)$ around the puncture at $z$. This should be captured in a potential localization formula for the measure $\mathrm{DH}_{\connections_{\tilde Q} \times \orbit_{\tilde \alpha}}$. We will leave this for future investigation.

\section{Further Comments}

In this final section we lay out some brief informal comments about how other well known results about Chern-Simons theory on circle bundles can be rephrased in the Caloron picture. We will discuss the canonical quantization of the Caloron theory. Furthermore, since we constructed the theory using bundle gerbe techniques, there are a lot more tools available for us to investigate it. Namely, bundle gerbe modules and their (twisted) K-theory c.f \cite{Bouwknegt:2002}. This will give us a new understanding of the $\mathrm{CS/WZW}$ correspondence in this 2-dimensional setting. We will introduce a certain class of $D$-branes in our 2d theory following Fuchs-Nikolaus-Schweigert-Waldorf \cite{Fuchs:2010}, and show $\gauge_Q$ invariance of the action in the presence of certain boundary $D$-brane terms. The branes we will consider will carry representations of the loop group at level $k$, in close analogy to the symmetric D-branes of WZW theory that wrap quantized conjugacy classes in $G$.

It is also worthwhile to note that topological gauge theories with gauge groups of the form $\BT \ltimes L\gauge$, where $\gauge$ is the gauge group of some ordinary finite dimensional gauge theory have been considered in \cite{Baulieu:1998}\footnote{We thank V. Pestun for bringing this reference to the author's attention}. In these theories the expectation values of observables are shown to compute certain $K$-theoretic intersections on the corresponding classical moduli space. It would be interesting to see if this has connections with our work. 

\subsection{Canonical Quantization} \label{CanonicalQuantizationSection}

The canonical quantization of Chern-Simons theory is will known, \cite{Axelrod:1991}. Here we sketch out how the moduli space of flat connections is realized as a symplectic reduction in the Caloron picture. We begin by recalling the following,

\begin{proposition}
The classical equations of motion for Caloron BF theory are given by (dependent equations)
\begin{equation}
(1) \quad F_L = d\lambda\, V, \qquad (2) \quad d_L V = 0
\end{equation}
\end{proposition}
Obviously $(1) \Rightarrow (2)$. Furthermore, we have $d_{\tilde L} \tilde V = 0 \Leftrightarrow d_L V = 0$, since $d_{\tilde L} \tilde V = (0,-\nabla_L \Phi, k \llangle \nabla_L \Phi, \Phi\rrangle )$ where $(0,-\nabla_L \Phi) := d_L(1,-\Phi)$.
We now assume that our Riemann surface has a boundary $\partial \Sigma = C \cong S^1$. The degree 2 equation of motion obviously becomes vacuous when restricted to the boundary, however, we must still have $d_L V = 0$. We denote by $ \CE$ the set of $( L, V=(1,-\Phi))$ when restricted to the boundary. We can define the symplectic form on $\CE$
\begin{eqnarray}
\Omega &= & \int_C \lllangle \delta \tilde L, \delta \tilde V \rrrangle\\
&=& k \int_C \left(  \llangle \delta \Phi, \Phi\rrangle \delta \lambda - \llangle \delta \Lambda, \delta \Phi \rrangle \right)
\end{eqnarray}
where $\tilde L$ is any lift of $L$. Note that this form does not depend on the lift.
The following proposition is follows immediately
\begin{proposition}
The above symplectic form is ${\gauge}_C$ invariant, and satisfies the moment-map equation (for $\eps  = (\theta, \gamma) \in \liegauge_C$)
\begin{eqnarray}
\iota_{\eps} \Omega &=& - k \int_C \lllangle \tilde \eps, d_{\tilde L} \tilde V \rrrangle \\
&=&  k \int_C    \llangle \nabla_L \Phi,  \gamma -\theta \,\Phi \rrangle \\
\end{eqnarray}
where $ \tilde \eps$ is any lift of $\eps$ to $\tilde \liegauge_C$. Thus the gauge group action is Hamiltonian, with moment map $\mu = - k\, d_{ \tilde L} \tilde V$, and so we can realize the classical phase space as the symplectic reduction
\begin{equation}
 \CM = \CE /\!/\, \gauge_C.
\end{equation}
\end{proposition}
Note that we view $d_{ \tilde L} \tilde V$ as an element of the dual $\liegauge_C^\vee$ by the above pairing.
A pre-quantum line bundle $\tilde \CL$ over $\CM$ can be defined in the following way. Consider the action of $\varphi \in \gauge_C$ on $\CL := \CE \times \BC^\times$
\begin{equation}
\varphi \cdot ( (L,V) , z) = ( \varphi\cdot(L,V) , \hol_{\ell(\varphi)}(C)\cdot z )
\end{equation}
where $\ell(\varphi)$ is the line bundle with connection over $C$ defined by $g^*_{\varphi}( \tilde G_k \to \check G ) \in \check H^2_{\check G}(Q) \cong \check H^2(C) \ni \ell(\varphi)$, where $g_{\varphi}$ is defined by $\varphi(q) = q.g_{\varphi}(q)$. Note that $\varphi \mapsto \hol_{\ell(\varphi)}(C) \in \BC^\times$ is a character of $\gauge_C$, independent of $(L,V)$. This is different to the case of the canonical quantization of Chern-Simons theory, where a cocycle is found. In this way, this bundle descends to a line bundle $\tilde \CL$ on $\CM$. To quantize this theory, in contrast to the usual case of quantization of Chern-Simons using Kahler polarizations, we are going to use real polarizations. We consider the following connection on $\tilde \CL$
\begin{equation}
\nabla = \delta + \int_C \lllangle \tilde L, \delta \tilde V \rrrangle,
\end{equation}
and we note that this connection is independent of the lift $\tilde L$. We easily see that $F_\nabla = \Omega$. We choose our real polarization of $\CE$ be given by the planes of constant $V$. Our classical Gauss law constraint $d_{\tilde L} \tilde V = 0$, becomes the quantum wave-function constraint
\begin{equation}
d_{\tilde L(t)} \frac{\delta}{\delta \tilde L(t)} \Psi(\tilde L) = 0,
\end{equation}
where $\Psi : \CM \to \tilde \CL$ is a section of our pre-quantum line bundle. We can find some canonical flat sections of $\CL^{\otimes k} \to \CE$ as follows. Pick a lift $(\tilde Q, \tilde L)$ of $(Q,L)$, then $\hol_{\tilde L}(C) \in \tilde G_k$. Let $\chi_\lambda$ be the normalized character of a level-$k$ integrable heighest weight module of $\tilde G_k$ (see \cite{CFT}), then it is not hard to check that
\begin{equation}
\Psi_\lambda : L \mapsto (L,\chi_\lambda( \hol_{\tilde L}(C) ))
\end{equation}
is a flat section of $\tilde \CL^k$. 
Now, note that 
\begin{equation}
\hol_{\tilde L}(C) = \left( \hol_\kappa(C) ,\ldots,\ldots \right) \in \tilde G.
\end{equation}
Thus $\Psi_\lambda$ should be interpreted as a formal power series in 
\begin{equation}
\tilde q = e^{2\pi \Ri \tilde \tau} := \hol_\kappa(C) = e^{2\pi \Ri\oint_C s^* \kappa }   \in S^1.
\end{equation}
To interpret this expansion parameter, we recall the usual definition for the complex modulus $\tau$ of an elliptic curve $E$.
Let $E$ be defined by $y^2 = f(x)$.  The canonical holomorphic 1-form on $E$ is given by $\lambda = dx/y$. We can recover the modulus via the period formula 
\begin{equation}
\tau = \frac{\oint_b \lambda}{\oint_a \lambda} \in \BH
\end{equation}
where we have chosen fundamental cycles $a$ and $b$ on $E$. This is used to establish an isomorphism $\BC/\BZ\langle1,\tau\rangle \to E$.

On the other hand, Let us consider the bundle $P := Q \times_{\check \AG} S^1 \to C$, the circle bundle over the boundary $C$, equipped with the one form $\kappa$. $P$ is topologically a torus. We have an obvious candidate for the $a$-cycle,  given by any single orbit $a$ of the $S^1$ action on $P$. There we find $\oint_a \kappa = 1$. Now, for a choice of $b$ cycle, we need to pick a section $s: C \to P$, and let $b = s(C)$, then we find $\oint_b \kappa = \oint_C s^* \kappa = \log \hol_\kappa (C)$. So we find the natural analogue of the modulus $\tau$ for $(P,\kappa)$ to be
\begin{equation}\label{tauRATIO}
\tilde \tau = \frac{\oint_b \kappa}{\oint_a \kappa} \in \BR.
\end{equation}
Thus the `modular' parameter $\tilde q = e^{2\pi \Ri \tilde \tau}$ that appears in our wave functions should properly be compared with usual parameter $q=e^{2\pi \Ri \tau}$ as it appears in the wave functions for the CS/WZW states, as explored in \cite{Falceto:1994}. Lastly, recall that the modular domain is given by $\tau \in \BH$ (or $q \in D$). Its circular boundary $\tilde q \in S^1$ (so called `Thurston' boundary) can be described by certain foliations, see e.g. \cite{Andersen:1998, Weitsman:1991}. The flat sections of the bundle with respect to $\kappa$ give us such a foliation of $\Sigma$, and it would be natural to suggest that the wave functions in our quantization are made convergent by a suitable continuation into the interior of the modular domain, or a suitable complexification of the connection $L$. This will be investigated further in a future work.

\subsection{Abelianization of the Partition Function}

We recall here some characteristic expressions of the partition function of Chern-Simons theory on circle bundles. As was discussed earlier, Beasley-Witten showed that the path integral admits an expression as a symplectic integral. The expression they found can be schematically expressed as
\begin{equation}
Z^{CS}_k(M_{(p,g)}) = \sum_{\ell} \int_{\Fg \times \CM_\ell} \CF_\ell
\end{equation}
where $\ell$ indexes certain classes of classical solutions to equations similar to the Yang-Mills equations. We have elaborated how such a symplectic localization formula appears naturally in the Caloron setting. Another type of expression for the path integral is found in the works of Aganagic-Saulina-Ooguri-Vafa \cite{Aganagic:2005}, and Blau-Thompson \cite{Blau:2006}, where the path integral can be localized on to the constant modes of an adjoint scalar, yielding a matrix model of the form
\begin{equation}
Z^{CS}_k(M_{(p,g)}) = \sum_{n \in \BZ_p^{r} } \int_{\Ft} [d\phi] f(\phi) \exp  \Ri (k + c_{\Fg}) \left( \langle n, \phi \rangle + 4 \pi p \langle \phi, \phi \rangle \right)
\end{equation}
where $f_g(\phi)$ is invariant under the affine Weyl group $\hat W$.
We give a brief explanation for how such a formula follows immediately from our Caloron formalism. Recall the formula for the path integral
\begin{equation}
Z = \int [DV][DL] e^{ 2 \pi i \int \lllangle F_{\tilde L}, \tilde V \rrrangle}.
\end{equation}
We can use the gauge group $\gauge_{\tilde Q}$ to gauge fix our adjoint scalar $\tilde V = (1,-\Phi, \frac{k}{2} \llangle \Phi, \Phi \rrangle) \in \tilde \Fg_k$ to be of the form
\begin{equation}
\tilde V_\phi = (1, -\phi, \frac{k}{2} \langle \phi, \phi \rangle ), \quad \phi \in \Fh.
\end{equation}
In terms of this abelianized adjoint field, the action will be of the form
\begin{eqnarray}
Z &=& \int [D\phi][DL] H(L,\phi) e^{ 2 \pi i \int \lllangle F_{\tilde L}, \tilde V_\phi \rrrangle} 
\end{eqnarray}
where $H(L,\phi)$ is the functional determinant of the gauge fixing. We can write the gauge fixed action in terms of the curvature $F_{\tilde L} = (d\lambda, f_L, r) $ as 
\begin{equation}
\int \lllangle F_{\tilde L}, \tilde V_\phi \rrrangle = \frac{k}{2}  \int d\lambda \langle\phi,\phi \rangle  + k \int \langle f_L , \phi \rangle + r.
\end{equation}
As happen in the usual case of abelianization, (see \cite{Blau:2006}), these terms should reduce to a matrix model of the form
\begin{equation}
Z = \sum_{\stackrel{\rightarrow}{n} \in \Omega} \int_\Fh [D\phi] H'(\phi) e^{ ik\left(  \frac{p}{4\pi} \langle \phi,\phi \rangle  +   \langle\stackrel{\rightarrow}{n}, \phi \rangle \right)} 
\end{equation}
where we have used $\int d\lambda =\frac{1}{2\pi} p$ the Chern class of the line bundle, and $\Omega$ is some lattice of magnetic charges for the curvature. In this way we easily see how such `abelianization' formulas for the path integral can be found. This also shows that the only dependence of the partition function on the Chern class $p$, is through the insertion of the operator
\begin{equation}
\CO(\phi) = e^{i \frac{(k+h^\vee)}{4\pi} p\langle \phi, \phi \rangle }.
\end{equation}

\subsection{The CS/WZW Correspondence} \label{CSWZWSection}

Here we will give a summary of how the usual interrelations between a bulk Chern-Simons theory and its coupling to a Wess-Zumino-Witten model on the boundary translate into the language of the Caloron theory when all the geometries involved are circle bundles.

We assume that we have the Caloron theory on a Riemann surface $\Sigma$ with a single boundary component $\partial \Sigma = C  \cong S^1$. As usual we denote the Caloron bundle over $\Sigma$ by $Q$,  On $C$, we construct a one-dimensional sigma model as follows. The space of fields are the sections of the bundle $Q$ over $C$
\begin{equation}
\CC = \Gamma( C ,  Q ).
\end{equation}
Now this space is empty if the bundle $Q$ is not trivializable over $C$, so we demand this as a requirement.
Now, for a particular lift $(\tilde Q, \alpha)$ of $(Q,L)$, consider the following sigma model action $S$ for $\gamma \in \CC$, defined by
\begin{equation}
\exp S_{(\tilde Q, \alpha)}(\gamma) := \hol_{\gamma(C)}( \tilde Q, \alpha)= \hol_{C}(\gamma^*(\tilde Q,\alpha)) \in \BC^\times, \quad \gamma \in \CC
\end{equation}
where $\alpha$ is the connection on the line bundle $\tilde Q \to Q$. We see that this action is simply the magnetic term for the worldline $\gamma(C)$ of a charged particle moving in the background gauge field $\alpha$ on $Q$. This action is the `Caloron' reduction of the topological term of the genus one WZW model.

To begin the analysis of this action, let us first consider a `small' gauge transformation $\tilde \varphi : \tilde Q \to \tilde Q$. Since these gauge transformations descend to automorphisms of $Q$, we have an action  on $\CC$.
We find $\tilde \varphi : (\gamma, \tilde Q, \alpha) \mapsto (\varphi(\gamma), \tilde Q, \tilde\varphi^*\alpha)$. The equivariance of $\alpha$ easily shows that the action is invariant,
\begin{equation}
S_{(\tilde Q, \alpha)}(\gamma) = S_{(\tilde Q, \tilde \varphi^*\alpha)}(\varphi(\gamma)).
\end{equation}
Secondly, lets consider a `large' gauge transformation, namely a twist of $\tilde Q$ by a line bundle $(\ell,\beta)$ on $\Sigma$. We have
\begin{equation}
\exp S_{(\tilde Q \otimes q^* \ell, \alpha\otimes 1 + 1 \otimes \beta)}(\gamma) = \exp S_{(\tilde Q , \alpha)}(\gamma) \times \hol_C(\ell,\beta)
\end{equation}
since $\gamma^* q^* = (q\gamma)^* = id$. Thus the amplitude $\exp S$ of the sigma model action acquires a factor of $\hol_C(\ell,\beta)$ when we perform a large transformation. This term can precisely cancel the anomaly inflow from the bulk Caloron theory in the following way. Let
\begin{equation}
S^{tot}_{\tilde Q}(\gamma, \tilde L,\tilde V) =  kS_{(\tilde Q, \alpha)}(\gamma) + S_{k,\tilde Q}^{\Cal}(\tilde L, \tilde V).
\end{equation}
We then have
\begin{equation}
(\ell,\beta)^* \exp S^{tot}_{\tilde Q} = \exp S^{tot}_{\tilde Q} \times  \hol^k_C(\ell,\beta) e^{- \tfrac{k}{2\pi} \int_\Sigma d\beta }.
\end{equation}
Now we use the classical Stokes theorem for a line bundle with connection $(\ell, \beta)$ on a surface $\Sigma$ with boundary $C$,
\begin{equation}
\hol_C(\ell,\beta) = e^{\tfrac{1}{2\pi} \int_\Sigma d\beta },
\end{equation}
to conclude that $\exp S^{tot}_{\tilde Q}$ is invariant under large gauge transformations.

\subsection{Branes In The Bulk }\label{BranesSection}

Other than the boundary terms that involve coupling to sigma models, as done in the previous section, we can also include $D$-brane boundaries (analogous to Chan-Paton factors in string theory). Here we give a brief summary of this using the language of gerbe morphisms as in \cite{Fuchs:2010}, which refer to without repeating definitions. The basic idea is that in a Caloron theory at level $k$, we have a class of $D$-branes labelled by the integral representations of $\tilde G$ at level $k$, on which our worldsheet $\Sigma$ can end. Let $\mathcal{I}_{\omega}$ be the denote the trivial bundle gerbe with characteristic form $\omega$.

\begin{definition} 
Let $\gerbe$ be a bundle gerbe over $B$. A $\gerbe$-$D$-brane $\dbrane = (C, \Phi)$ is a sub-manifold $C \subset B$ equipped with a 1-morphism (i.e. twisted vector bundle with compatible connection $A$) $\Phi : \gerbe|_{C} \to \mathcal{I}_{\omega}$ of gerbes over $C$. \end{definition}

\begin{proposition}[Fuchs, Nikolaus, Schweigert, Waldorf; \S 6, \cite{Fuchs:2010}] 
Let $\gerbe$ be a bundle gerbe over $B$, and $\dbrane$ a $\gerbe$-$D$-brane over $C$.
Given a smooth sub-manifold $\Sigma \to B$, with $\partial \Sigma \subset C$, and a 1-trivialization (i.e twisted line bundle with compatible connection) $\Psi : \gerbe|_\Sigma \to I_{\omega}$.
The product
\begin{equation}
\Hol{}_{\gerbe,\dbrane}(\Sigma) :=\Hol{}_{\gerbe}(\Sigma,\Psi) \Tr(\Hol{}_{A}({\partial \Sigma},\Psi)) ,
\end{equation}
is independent of the choice of trivialization.
\end{proposition}
The trace $\Tr(\Hol{}_{A}({\partial \Sigma},\Psi))$ is known as the $D$-brane holonomy of $\dbrane$ on $\partial \Sigma$ in the trivialization $\Psi$. In our case, we can construct a twisted vector bundle out of a positive energy representation $W$ of $\tilde G_k$, as outlined in \cite{Bouwknegt:2002}, with more details in appendix \ref{twistedmodule}.

\begin{proposition}
Let $\gerbe_Q$ be the lifting bundle gerbe of $Q$ over $B$. Given a smooth sub-manifold $\Sigma \to B$, and a choice of lift $\tilde Q \to \Sigma$, let $\dbrane_W^{\tilde Q}$ be the $\gerbe$-$D$-brane over all of $B$ given by the associated bundle to $\tilde Q$ in a representation $W$. The product
\begin{equation}
\Hol{}_{\gerbe_Q,\dbrane}(\Sigma) :=\Hol{}_{\gerbe_Q}(\Sigma,\tilde Q) \,\chi_W(\Hol{}_{\tilde L}({\partial \Sigma},\tilde Q))
\end{equation}
where $\chi_{W}$ is the \emph{normalized} character of the representation $W$, is $\gauge_{\tilde Q}$ invariant. Furthermore, it is independent of the choice of lift $\tilde Q$, and hence $\gauge_Q$ invariant.
\end{proposition}

However, this character is only convergent once we analytically continue the relevant modulus $\tilde \tau$ to have  positive imaginary part, as outlined in the comments \ref{tauRATIO}. Thus we can see that we have an isomorphism between the space of $D$-branes and the space of states in the canonical quantization of the Caloron theory on a circle. We will explore this class of $D$-branes in a future publication.

\section*{Appendix I: The Level $k$ Central Extension of $LG$}\label{LiftingGerbeAppendix}

For a simple real Lie group $G$ with Lie algebra $\Fg$ and non-degenerate inner product $\langle, \rangle$, we consider its loop group $LG = \Map(\BT, G)$, and the corresponding loop algebra $\flg = \Map(\BT, \Fg)$. We extend the structure of $LG$ by allowing rotation of the loops, to get $\LGS = \BT \ltimes LG$. The product on $\check \AG$ is 
\begin{equation} (\phi_1, \gamma_1).(\phi_2,\gamma_2) = (\phi_1+\phi_2, \gamma_1\rho_{\phi_1}(\gamma_2)).\end{equation}
where $\rho_{\phi}(\gamma)(\theta) = \gamma(\theta-\phi)$. 
We also define $\check\Fg = \Ft_R \oplus \flg = Lie(\LGS)$. 
Furthermore, $\flg$ has an inner product $\llangle \xi_1,\xi_2 \rrangle = \int_{\BT} \langle \xi_1(\theta),\xi_2(\theta)\rangle d\theta$, and a non-trivial 2-cocycle $ \omega(\xi_1, \xi_2) = \llangle \xi_1,  \xi_2' \rrangle$.
This defines a central extension $\affg = \flg \oplus \Ft_c$, which comes from a group extension $\BT \to \hat \AG \to LG$, known as the level 1 central extension. The $k$th tensor power of this central extension is the level $k$ extension.
The full affine Lie algebra $\affgs_k = \mathrm{Lie}(\tilde \AG_k)$ is given by including both of the above extensions of $\flg$, $\tilde \Fg_k = \Ft_R \oplus \flg \oplus \Ft_c$, with bracket
\begin{equation} [(x_1,\xi_1,y_1), (x_2,\xi_2,y_2) ] = (0,[\xi_1,\xi_2]+x_1 \xi_2'-x_2 \xi_1', k\llangle \xi_1,\xi_2' \rrangle) \end{equation}
We have a commuting diagram of Lie homs 
\begin{equation}
\begin{gathered}
\xymatrix{ 
 & \affg \ar[dr]\ar[dl]& \\
\affgs \ar[dr] & &\flg \ar[dl]\\
 & \check \Fg& \\
}
\end{gathered}
\end{equation}
where the arrows to the right are projections, and the left arrows are inclusions. For $\gamma$ in the identity component of $\LG$, we have
\begin{equation}\label{affinegadjoint}
\Ad_\gamma(x,\xi,y) = (x, \ad_\gamma \xi - x \gamma' \gamma^{-1}, y - \llangle \gamma^{-1}\gamma' , \xi\rrangle + \tfrac{1}{2} x \llangle \gamma^{-1}\gamma', \gamma^{-1}\gamma'\rrangle) 
\end{equation}
Furthermore, $\tilde \Fg_k$ is equipped with a non-degenerate invariant inner product: 
\begin{equation} \lllangle (x_1,\xi_1,y_1), (x_2,\xi_2,y_2) \rrrangle_k = k\llangle \xi_1, \xi_2 \rrangle - x_1y_2-x_2y_1 \end{equation}
The group cocycle $\sigma : G \times \affg \to \Ft_c$ of this extension is given by
\begin{equation} \sigma(\gamma,(x,\xi)) := \Ad_{\gamma}(x,\xi,0) -  (\Ad_{\gamma}(x,\xi),0) = \llangle -\xi +\tfrac{1}{2} x \,\gamma^{-1}\gamma', \gamma^{-1}\gamma'\rrangle  \end{equation}

\section*{Appendix II: The Lifting bundle gerbe}\label{liftingappendix}
Let $q:\tilde K \to K$ be a central extension of Lie groups by $\BT$, and let $P \to X$ be a $K$-bundle. Here we will describe the lifting bundle gerbe, denoted $\mathbb{L}_P^{\tilde K}$. This gerbe is represented by the following diagram of maps

\begin{equation}
\begin{gathered}
\xymatrix{ 
J \ar[r]_{\tilde \xi} \ar[d]_{\rho} & \tilde K \ar[d]_q \\
P^{[2]} \ar[r]_{\xi} \ar@{=>}[dr]_{\pi_i} & K \\
 & P \ar[d]_\pi \\
 & X
}
\end{gathered}
\end{equation}
where $P^{[2]} = P \times_X P$, and $\pi_{1,2}$ are the projections on to the factors, $\xi(p,pg) := g$, and $J = \xi^* \tilde K$. The lifting bundle gerbe $\mathbb{L}_P^{\tilde K}$ is alternatively described geometrically as the total space of $J \cong P \times \tilde K$. The bundle is given by $\rho : (p, \tilde k) \to (p, pk) \in P^{[2]}$. An essential part of the bundle gerbe is the groupoid multiplication map $\tilde m : J_{12} \times J_{23} \to J_{13}$ defined by
\begin{equation} \tilde m: (p, \tilde k_1) \times (pk_1, \tilde k_2) \mapsto (p, \tilde k_1\tilde k_2) \end{equation}

For the connective structure on this gerbe, we follow Gomi \cite{Gomi:2003}. Let $\mu,\tilde\mu$ be the Maurer-Cartan forms on $K,\tilde K$ respectively, and let $\sigma : \Fk \to \tilde \Fk$ a splitting (as vector spaces) of the corresponding Lie algebras. Often there will be a canonical choice of splitting $\sigma$, but it is important to determine the dependence on this choice. Then we see that $\nu_\sigma := \tilde\mu - \sigma(q^*\mu) \in \forms^1(\tilde K, \Ft)$ is a connection for the free $\BT$ action on $\tilde K$, with curvature $F_{\nu_\sigma} = -\tfrac{1}{2} \omega_\sigma(\mu,\mu)$, where $\omega_\sigma : \Fk \times \Fk \to \Ft$ is the Lie algebra cocycle $\omega_\sigma(X,Y) = [\sigma(X),\sigma(Y)]_{\tilde \Fk} - \sigma([X,Y]_\Fk)$. 
Furthermore, let $Z_\sigma : K \times \Fk \to \Ft$ be the group cocycle associated to $\sigma$, $Z_\sigma(k,X) = \Ad_k \sigma(X) - \sigma(\Ad_k(X))$. Obviously $\tilde\xi^*\nu_\sigma$ is a connection on $J$, but it is not a bundle gerbe connection. We need to find a modification $\varepsilon \in \forms^1(P^{[2]}, \Ft)$, such that $\nabla = \tilde\xi^*\nu_\sigma + \rho^* \varepsilon$ is a bundle gerbe connection, i.e, 
\begin{equation} \tilde \pi_{12}^* \nabla + \tilde \pi_{23}^* \nabla = \tilde m^* \tilde \pi_{13}^* \nabla \end{equation}

\begin{theorem}[Gomi]
The form $\varepsilon = Z_\sigma(\xi^{-1},\pi_1^*A)$
solves the above equation, i.e. the connection
\begin{equation} \nabla_{A,\sigma} = \tilde\xi^*\nu_\sigma + \rho^*Z_\sigma(\xi^{-1}, \pi_1^* A) \end{equation}
is a bundle gerbe connection.
\end{theorem}

To find a curving for this connection, we need to introduce a \emph{bundle splitting}. This is a map $s : P \times_{K} \tilde \Fk \to P \times \Ft$ (Gomi uses the letter $L$) that is identity on the central component. Let $s$ be a bundle splitting of $P$, and consider the form $s(\sigma( F_{A}))  \in \forms^2(P,\Ft)$
where $F_A$ is the curvature of the connection $A \in \forms^1(Y,\Fk)$.
Then we have
\begin{theorem}[Gomi]
The curvature of the connection $\nabla_{A,\sigma}$ is given by
\begin{equation} F_{\nabla_{A,\sigma}} = (\pi_1^* - \pi_2^*)( \tfrac{1}{2}\omega_\sigma(A,A) +  s(\sigma( F_{A}))) \end{equation}
I.e.
\begin{equation} f_{A,\sigma,s} = -( \tfrac{1}{2}\omega_\sigma(A,A) +  s(\sigma( F_{A}))) \end{equation}
is a curving for $\nabla_{A,\sigma}$.
\end{theorem}

This completes the construction of the Brylinski-Gomi curving, our bundle gerbe is $\mathbb{L}_P^{\tilde K}(A,\sigma)$ with connection $\nabla_{A,\sigma}$ and curving $f_{A,\sigma,s}$.

\subsection*{Lifts as Trivializations}

The following result expresses the basic fact that the lifting bundle gerbe topologically represents the obstruction to finding a lift of the structure group. Consequentially, any particular lift provides a trivialization of the lifting bundle gerbe.
\begin{proposition}
Choose a lift $\tilde p:(\tilde P,\tilde A) \to( P,A)$, as a $\tilde K$ bundle with connection. Then this data provides us with a trivialization of gerbes $\tau_{\tilde P} : P \to \mathcal{I}_{\varrho}$, where $\pi^* \varrho = s(F_{\tilde A})$
\end{proposition}
\begin{proof}
Since $J = P \times \tilde K$, we have the isomorphism of lines over $P^{[2]}$
\begin{equation} J \otimes \pi_2^* \tilde P \to  \pi_1^* \tilde P \quad  :\quad (p,\tilde k, \widetilde{pk} ) \mapsto (\widetilde{pk}\tilde k^{-1} ) \end{equation}
where $ \widetilde{pk} \in \tilde P$ is a lift of $pk \in P$. The connection $a = \tilde A - {\tilde p}^*\sigma(A)$ on $\tilde P \to P$ extends to a gerbe connection  $\delta a$ on $\delta \tilde P$. To show that this connection is compatible with $\nabla_{A,\sigma}$, we compute
\begin{eqnarray}
  R_{\tilde \xi}^* a &=&  (\Ad_{\tilde \xi^{-1}} \tilde A + \tilde \xi^* \tilde \mu) - \tilde p^*( \sigma(\Ad_{\xi^{-1}}A) + \sigma (\xi^*\mu)) \\
&=& \Ad_{\tilde \xi^{-1}} a + \nu_{\sigma}  + \tilde p^* Z_\sigma(\xi^{-1}, A) \\
& =& \Ad_{\tilde \xi^{-1}} a + \nabla_{A,\sigma}
\end{eqnarray}
This is equivalent to the compatibility between $a$ and $\nabla_{A,\sigma}$. The two form $da$ is a canonical choice of curving for $\delta a$. So we find that the difference in curvings is
\begin{equation} f_{A,\sigma,s} - da =   -\tfrac{1}{2} \omega_\sigma(A,A) - s(\sigma(F_A))-da \end{equation}
Now, if we write $\tilde A = \sigma(A) + a $, then 
\begin{eqnarray}
F_{\tilde A} &=& \sigma(dA) +da + \tfrac{1}{2}[\sigma(A),\sigma(A)] \\
&=& \sigma(F_A) + da + \tfrac{1}{2}\omega_{\sigma}(A,A) 
\end{eqnarray}
Thus we see that
\begin{equation}  f_{A,\sigma,s} - da= -s(F_{\tilde A}) \end{equation}
\end{proof}

\subsection*{The Principal Module Connection}\label{twistedmodule}
Here we include some new constructions for lifting bundle gerbes. The goal is the construction of certain modules for the lifting bundle gerbe, as defined in \cite{Bouwknegt:2002}.
Let $V \in \Rep_0(\tilde K)$ be a representation such that the induced representation of the central $\BT$ has weight one.  Consider the trivial vector bundle $E = P \times V$ over $P$, and gerbe action map $\phi : J \times \pi_2^* E \to \pi_1^* E$, 
\begin{equation} \Phi : (p, \tilde k) \times (pk, v) \mapsto (p, \tilde k v) \end{equation}
It is easily seen that this is a bundle gerbe module.
This establishes a map $\Rep_0(\tilde K) \to \Mod( \mathbb{L}_P^{\tilde K})$. 
Note that a trivialization (i.e. a lift) $\tilde P \to P$ gives descent data for $E$, i.e. $E \cong \pi^* (\tilde P \times_{\tilde K} V)$.
Here we extend Gomi's construction to show the existence of a module connection on $E$ compatible with the connection $\nabla_{A,\sigma}$ defined in the last section. We will find a suitable analogue of the notion of a connection on a principal bundle, which will naturally provide a covariant derivative on all associated bundles. We begin by looking for a compatible \emph{principal} connection on the principal ${\tilde K}$ bundle $j = \pi_1\rho : J \to P$, where ${\tilde K}$ acts on the right. Clearly, $\tilde\xi^* \tilde \mu$ is a connection on this bundle, so we seek a modification form $\varepsilon \in \forms^1(P,\ad_j \tilde \Fk)$, such that
\begin{equation} \phi = \tilde\xi^* \tilde \mu + \rho^* \varepsilon \end{equation} is a bundle gerbe module connection on $J$, compatible with the connection $\nabla_{A,\sigma}$, i.e.
\begin{equation} \tilde \pi_{12}^* \nabla_{A,\sigma} + \tilde \pi_{23}^* \phi = \tilde m^* \tilde \pi_{13}^* \phi \end{equation}

\begin{proposition}
The form $\varepsilon = \Ad_{\xi^{-1}} \sigma( \pi_1^* A) \in \forms^1(P,\ad_j \tilde \Fk)$ solves the above equation. I.e., the connection
\begin{equation} \phi_{A,\sigma} = \tilde \xi^* \tilde \mu + \rho^* \Ad_{\xi^{-1}}\sigma(  \pi_1^* A) \end{equation}
is a bundle gerbe module connection on $j:J \to P$. Furthermore, it can be expressed as
\begin{equation} \phi_{A,\sigma} = \rho^* \pi_2^* \sigma( A)+ \nabla_{A,\sigma}\end{equation} 

\end{proposition}
\begin{proof}
First of all, we show the equivalence of the two expressions.
Note that we have
\begin{eqnarray}
 \rho^* \pi_2^* \sigma( A)+ \nabla_{A,\sigma} &=& \rho^*\pi_2^*\sigma(A)+\tilde \xi^*\nu_\sigma +  \rho^*Z_\sigma(\xi^{-1}, \pi_1^*A)\\
 &=& \tilde \xi^* \tilde \mu - \rho^* \sigma(\xi^* \mu) + \rho^* \pi_2^* \sigma(A) +\rho^*Z_\sigma(\xi^{-1}, \pi_1^*A) \\
&=& \tilde \xi^* \tilde \mu + \rho^* ( -\sigma(\xi^* \mu) + \pi_2^* \sigma(A) +Z_\sigma(\xi^{-1}, \pi_1^*A) )
\end{eqnarray}
The $\rho^*$ term is
\begin{equation} \sigma(\Ad_{\xi^{-1}}  \pi_1^*A) +Z_\sigma(\xi^{-1}, \pi_1^*A) =  \Ad_{\xi^{-1}}\sigma( \pi_1^*A) \end{equation}
Thus the two expressions are equivalent.
To show that it is a bundle gerbe module connection, we use the second form to see
\begin{equation} \tilde \pi_{12}^* \nabla_{\theta,\sigma} + \tilde \pi_{23}^* \phi - \tilde m^* \tilde \pi_{13}^* \phi =  (\tilde \pi_{23}^* \rho^* \pi_2^*  - \tilde m^* \tilde \pi_{13}^*\rho^* \pi_2^*) \sigma( A)\end{equation}
because $ \nabla_{\theta,\sigma}$ is a bundle gerbe connection.
Now we have
\begin{equation} \tilde \pi_{23}^* \rho^* \pi_2^* - \tilde m^* \tilde \pi_{13}^*\rho^* \pi_2^* = (\pi_2 \rho \tilde \pi_{23})^*- ( \pi_2 \rho \tilde \pi_{13} \tilde m)^*= (\pi_2 \rho )^*\left((\tilde \pi_{23})^*- (  \tilde \pi_{13} \tilde m)^*\right)\end{equation}
which vanishes because of the properties of the multiplication $m$.
\end{proof}
If $\sigma, \sigma'$ are two splittings of the Lie algebra, then we have
\begin{equation} \phi_{A,\sigma'} - \phi_{A,\sigma} =  \rho^* \pi_1^*(\sigma'-\sigma)(A)\end{equation}
This follows easily from a result of Gomi, which reads
\begin{equation}\nabla_{A,\sigma'} - \nabla_{A,\sigma} = \rho^*(\pi_1^*-\pi_2^*)((\sigma'-\sigma)(A))\end{equation}
We now calculate the curvature of $\phi$. 
\begin{eqnarray} 
F_\phi &=& d\phi + \tfrac{1}{2}[\phi \wedge \phi] \\
&=& \rho^* \pi_2^* \left\{ \sigma(dA) +  \tfrac{1}{2}[\sigma(A) \wedge \sigma(A)] \right\} + F_\nabla \\
&=& \rho^* \pi_2^* \left\{ \sigma(F_A) + \tfrac{1}{2} \omega_\sigma(A,A) \right\} + F_\nabla
\end{eqnarray}
However, we know that
\begin{equation} F_\nabla = \rho^* (\pi_1^* -\pi_2^*)( \tfrac{1}{2} \omega_\sigma(A,A) + s(\sigma( F_{A}))) = \rho^* (\pi_1^* -\pi_2^*) (- f_{A,\sigma,s})\end{equation}
where $f_{A,\sigma,s} = -\tfrac{1}{2} \omega_\sigma(A,A) -s(\sigma( F_{A}))  $ is the curving. So we find

\begin{equation}\label{Fphi}
F_\phi = \rho^* \pi_2^* \left\{ \sigma(F_A) \right\} + \rho^* \pi_1^* (\tfrac{1}{2} \omega_\sigma(A,A)) + \rho^* (\pi_1^* -\pi_2^*)s(\sigma( F_{A})) 
\end{equation}

An important construction for modules is that of the module curvature, which is completely analogous to the curvature of a connection on a vector bundle.
\begin{definition}[\cite{Bouwknegt:2002}]
For a bundle gerbe $(\rho : J \to Y^{[2]}, \pi :Y\to X, \nabla, f)$ and module $(E\to Y, \nabla_E )$, the bundle gerbe \emph{module curvature} is the form $\mathcal{F}_E \in \forms^2(Y, \End(E)) $ defined by
\begin{equation*} \rho^* \pi_2^* \mathcal{F}_E := F_E + \pi_1^* f\, \mathrm{Id}_E
\end{equation*}
\end{definition}
In the case of a lifting bundle gerbe $\mathbb{L}_P^{\tilde K}$, we can define the \emph{principal curvature} $\mathcal{F}_\phi \in \forms^2(X, \Ad_P {\tilde \Fk})$ by
\begin{equation*} 
\rho^* \pi_2^* \mathcal{F}_\phi := F_{\phi} + \pi_1^* f_{A,\sigma,s}
\end{equation*}
Using \ref{Fphi}, we see $F_\phi + \pi_1^* f_{A,\sigma,s}  = \rho^* \pi_2^* (\sigma(F_A) - s(\sigma( F_{A})))$, so the principal curvature is
\begin{equation} \mathcal{F}_\phi =(1 - s)\sigma( F_{A})\in \forms^2(P, \Ad \tilde \Fk) .\end{equation}
We can check that $\mathcal{F}_\phi$ is equivariant under ${\tilde K}$, i.e. $R_{\tilde\gamma}^* \mathcal{F}_\phi = \Ad_{\tilde\gamma^{-1}} \mathcal{F}_\phi$. We can also see that $\mathcal{F}_\phi \in \ker(s)$.
Using this, we can easily compute this curvature for the loop space lifting bundle gerbe $\mathbb{L}_k(L,V) = \mathbb{L}_Q^{\tilde \AG_k}$ with the connective structure described previously.
\begin{lemma}
In the case of the loop space lifting bundle gerbe $\mathbb{L}_k(L,V)$, the module curvature of the principal module is given by
\begin{equation} \mathcal{F}_\phi = \left(d\lambda,F_\Lambda, -k\llangle F_\Lambda, \Phi \rrangle + \tfrac{k}{2} d\lambda\llangle \Phi,\Phi \rrangle \right) \end{equation}
\end{lemma}

\end{onehalfspace}

\bibliographystyle{unsrt}
\bibliography{./CaloronBF.bbl}

\end{document}